\documentclass[12pt,a4paper]{amsart}

\usepackage[utf8]{inputenc}
\usepackage{enumerate,bbm,titletoc,a4,mathtools,amssymb,xspace,amsthm,esvect}
\usepackage[usenames,dvipsnames]{xcolor}
\usepackage[pagebackref,colorlinks,linkcolor=BrickRed,citecolor=CadetBlue,urlcolor=black,hypertexnames=true]{hyperref}
\usepackage[all]{xy}

\DeclareMathOperator{\id}{id}
\DeclareMathOperator{\dom}{dom}
\DeclareMathOperator{\ran}{ran}
\DeclareMathOperator{\re}{Re}
\DeclareMathOperator{\spa}{span}
\DeclareMathOperator{\tr}{tr}

\DeclareMathOperator{\opm}{M}

\newcommand{\her}{{\rm sa}}
\newcommand{\ve}{\varepsilon}

\newcommand{\bA}{\mathbb{A}}
\newcommand{\bP}{\mathbb{P}}
\newcommand{\N}{\mathbb{N}}
\newcommand{\Z}{\mathbb{Z}}
\newcommand{\R}{\mathbb{R}}
\newcommand{\C}{\mathbb{C}}
\newcommand{\cA}{\mathcal{A}}
\newcommand{\cB}{\mathcal{B}}
\newcommand{\cC}{\mathcal{C}}
\newcommand{\cD}{\mathcal{D}}
\newcommand{\cE}{\mathcal{E}}
\newcommand{\cF}{\mathcal{F}}
\newcommand{\cH}{\mathcal{H}}

\newcommand{\cK}{\mathcal{K}}
\newcommand{\cM}{\mathcal{M}}
\newcommand{\cN}{\mathcal{N}}

\newcommand{\ux}{\underline x}
\newcommand{\uu}{\underline t}
\newcommand{\uX}{\underline X}
\newcommand{\uY}{\underline Y}

\newcommand*{\matc}[1]{\opm_{#1}(\C)}
\newcommand*{\sos}[1]{\Sigma^2 #1}
\newcommand*{\soc}[1]{[#1,#1]}

\newcommand{\Langle}{\mathop{<}\!}
\newcommand{\Rangle}{\!\mathop{>}}
\newcommand{\mx}{\Langle \ux\Rangle}
\newcommand{\rx}{\R\!\mx}
\newcommand*{\rxd}[1]{\R\!\mx\!{}_{#1}}
\newcommand{\bi}{\R\!\Langle x_1,x_2\Rangle}
\newcommand{\bih}{\R\!\Langle x_1,x_2\Rangle\!_\her}
\newcommand{\rxh}{\R\!\mx\!{}_\her}

\newcommand{\cx}{\C\!\mx}
\newcommand{\ser}{\R\!\Langle\!\!\mx\!\!\Rangle}

\makeatletter
\def\moverlay{\mathpalette\mov@rlay}
\def\mov@rlay#1#2{\leavevmode\vtop{
		\baselineskip\z@skip \lineskiplimit-\maxdimen
		\ialign{\hfil$#1##$\hfil\cr#2\crcr}}}
\makeatother

\newcommand{\plangle}{\moverlay{(\cr<}}
\newcommand{\prangle}{\moverlay{)\cr>}}
\newcommand{\rat}{\R\plangle \ux \prangle}

\newcommand{\nc}{nc\xspace}
\newcommand{\ncmeasure}{noncommutative joint distribution\xspace}
\newcommand{\ncmeasures}{noncommutative joint distributions\xspace}
\newcommand{\Ncmeasures}{Noncommutative joint distributions\xspace}
\newcommand{\integrable}{power-integrable\xspace}

\newcommand{\integ}{L^\omega(\cF,\tau)}

\usepackage[margin=2.5cm]{geometry}

\makeatletter
\newtheorem*{rep@thm}{\rep@title}
\newcommand{\newreptheorem}[2]{%
\newenvironment{rep#1}[1]{%
 \def\rep@title{#2 \ref{##1}}%
 \begin{rep@thm}}%
 {\end{rep@thm}}}
\makeatother

\newtheorem{thm}{Theorem}[section]
\newtheorem{lem}[thm]{Lemma}
\newtheorem{cor}[thm]{Corollary}
\newtheorem{prop}[thm]{Proposition}
\newtheorem{thmA}{Theorem}

\theoremstyle{definition}

\newtheorem{exa}[thm]{Example}
\theoremstyle{remark}
\newtheorem{rem}[thm]{Remark}

\makeatletter
\newcommand{\mycontentsbox}{%
	{\centerline{NOT FOR PUBLICATION}
		\addtolength{\parskip}{-2.3pt}
		\tableofcontents}}
\def\enddoc@text{\ifx\@empty\@translators \else\@settranslators\fi
	\ifx\@empty\addresses \else\@setaddresses\fi
	\newpage\mycontentsbox\newpage\printindex}
\makeatother

\numberwithin{equation}{section}

\title[Trace-positive polynomials and the unbounded tracial moment problem]{
Globally trace-positive noncommutative polynomials
and the unbounded tracial moment problem
}

\author[I. Klep]{Igor Klep${}^1$}
\address{Igor Klep, Faculty of Mathematics and Physics, University of Ljubljana}
\email{igor.klep@fmf.uni-lj.si}
\thanks{${}^1$Supported by the Slovenian Research Agency grants J1-2453, N1-0217, J1-3004 and P1-0222.}

\author[C. Scheiderer]{Claus Scheiderer${}^2$}
\address{Claus Scheiderer, Department of Mathematics and Statistics, University of Konstanz}
\email{claus.scheiderer@uni-konstanz.de}
\thanks{${}^2$Supported by DFG grant SCHE281/10-2.}

\author[J. Vol\v{c}i\v{c}]{Jurij Vol\v{c}i\v{c}${}^3$}
\address{Jurij Vol\v{c}i\v{c}, Department of Mathematics, Drexel University}
\email{jurij.volcic@drexel.edu}
\thanks{${}^3$Supported by Villum Fonden via the Villum Young Investigator
grant (No. 37532), the National Science Foundation grant DMS-1954709, and the Slovenian Research Agency grants J1-3004 and P1-0222.}

\date{\today}

\begin{document}

\begin{abstract}
A noncommutative (\nc) polynomial is called (globally) trace-positive 
if its evaluation at any tuple of operators in a tracial von Neumann algebra has nonnegative trace.
Such polynomials emerge as trace inequalities in several matrix or operator variables, 
and are widespread in mathematics and physics.
This paper delivers the first Positivstellensatz for global trace positivity of \nc polynomials.
Analogously to Hilbert's 17th problem in real algebraic geometry,
trace-positive \nc polynomials are shown to be weakly sums of hermitian squares and commutators of regular \nc rational functions.
In two variables, this result is strengthened further using 
a new sum-of-squares certificate with concrete univariate denominators for nonnegative bivariate polynomials.

The trace positivity certificates in this paper 
are obtained by convex duality through solving the so-called
unbounded tracial moment problem,
which arises from noncommutative integration theory and free probability.
Given a linear functional on \nc polynomials, the tracial moment problem asks whether it is a joint distribution of integral operators affiliated with a tracial von Neumann algebra.
A counterpart to Haviland's theorem on solvability of the tracial moment problem is established.
Moreover, a variant of Carleman's condition is shown to guarantee the existence of a solution 
to the tracial moment problem.
Together with semidefinite optimization, this is then used to prove that
every trace-positive \nc polynomial admits 
an explicit approximation in the 1-norm on its coefficients
by sums of hermitian squares and commutators of \nc polynomials.
\end{abstract}

\maketitle


\section{Introduction}

Trace inequalities in several operator variables are ubiquitous in mathematics and physics. 
For example, Golden--Thompson, Lieb--Thirring inequalities and their generalizations play an important role in quantum statistical mechanics \cite{Car,SBT}.
Another source of trace inequalities is quantum information \cite{Bei,PhysRev}, where they materialize through data processing inequalities and restrictions on quantum correlations.
In operator algebras \cite{NT,FNT} and noncommutative probability \cite{GS,JZ}, they appear through H\"older, Minkowski and other inequalities in noncommutative $L^p$-spaces, as well as trace convexity of entropy.
In mathematical optimization, hierarchies of semidefinite programs based on tracial noncommutative optimization are applied to compute matrix factorization ranks \cite{GdLL}.
Finally, the recently resolved \cite{JNVWY} Connes' embedding problem has several interpretations in terms of trace inequalities \cite{Had,Rad,KS,Oza}.
This paper studies trace polynomial inequalities that are valid globally, without restrictions on the variables or their norms, in all finite von Neumann algebras. 
For the first time, 
we provide necessary and sufficient certificates (Positivstellens\"atze) for such inequalities to hold, 
obtained by solving the associated unconstrained tracial moment problem.

Let $\rx$ be the real free $*$-algebra of noncommutative (\nc) polynomials in self-adjoint variables $x_1,\dots,x_n$.
A tracial von Neumann algebra is a pair 
$(\cF,\tau)$ consisting of a finite von Neumann algebra $\cF$ and a tracial state $\tau$ on $\cF$. 
Given an \nc polynomial $f\in\rx$ and a tuple of self-adjoint operators $\uX\in\cF_\her^n$, 
we consider the evaluation $f(\uX)\in\cF$ and its trace $\tau(f(\uX))\in\C$. 
We say that $f\in\rx$ is {\bf (globally) trace-positive} if 
$\tau(f(\uX))\ge0$ for all $\uX\in\cF_\her^n$ and 
all tracial von Neumann algebras $(\cF,\tau)$.

The notion of trace positivity of \nc polynomials fits in between 
positivity (of commutative polynomials) on tuples of real numbers 
and positive semidefiniteness on tuples of operators. 
Commutative and operator positivity are both well-studied, 
the former in real algebraic geometry and moment problems \cite{CF,Mar,Sche,Schm1}, 
and the latter under the umbrella of free analysis and free convexity \cite{HM,ANT,GKVW}. 
In both cases, algebraic certificates for positivity, 
the so-called Positivstellens\"atze, are usually given in terms of sums of squares. 
For example, 
the famous resolution of Hilbert's 17th problem asserts that a commutative polynomial is nonnegative 
if and only if it is a sum of squares of {\it rational} functions. 
Analogously, an \nc polynomial is everywhere positive semidefinite if and only if it is a sum of hermitian squares of \nc polynomials \cite{Hel,McC}.
On the other hand, results on trace-positive \nc polynomials are scarcer.
Only algebraic certificates for trace positivity on domains with additional restrictions have been given so far.
In \cite{KS,KMV}, Positivstellens\"atze for trace positivity on bounded domains (in terms of Archimedean quadratic modules) were derived. 
Likewise well-understood are \nc polynomials that are trace-positive on tuples of $k\times k$ matrices for a fixed $k\in\N$ \cite{KSV}, 
which were focal in the Procesi--Schacher conjecture \cite{PS}. 
On the other hand, a characterization of \nc polynomials that are trace-positive without norm or dimension restrictions on the input has been absent (except for the single operator case,
where even positivity of polynomials in trace powers can be handled \cite{KPV1}).

Polynomial positivity is dual to moment problems \cite{CF,PV,Schm1}.
While the classical moment problem aims at determining which functionals on a polynomial ring arise from integration with respect to some measure, 
the tracial moment problem considers functionals on a free algebra that arise from noncommutative integration \cite{Nel} with respect to a trace on a von Neumann algebra.
Its $C^*$-analog pertains to relativistic quantum theory \cite{Dub}.
In free probability, a unital algebra $\cA$ with a trace $\tau$ is viewed as a noncommutative probability space, 
and the functionals on $\rx$ of the form $p\mapsto \tau(p(\uX))$ for $\uX\in\cA^n$ 
are called noncommutative laws or joint distributions of $\uX$ \cite{VDN,MS}.
\Ncmeasures given by evaluations on tuples of elements from tracial von Neumann algebras $(\cF,\tau)$ have been characterized by the solution of the bounded tracial moment problem \cite{Had,Rad}.
The special case of \ncmeasures of tuples from finite-dimensional von Neumann algebras is settled in \cite{BKP16}.
This paper solves the unbounded tracial moment problem of describing functionals arising from $\tau$ and evaluations on (possibly unbounded) operators affiliated with $\cF$.

\subsection{Main results}

The contribution of this paper is twofold: we solve the unbounded tracial moment problem, 
and derive a Positivstellensatz for trace-positive \nc polynomials. 
Before stating these results, we require some notation.
Given a tracial von Neumann algebra $(\cF,\tau)$, let $\integ$ be its $*$-algebra of \integrable operators, 
i.e., the intersection of all the $L^p$-spaces associated with $(\cF,\tau)$.
That is, $X\in \integ$ if $X$ is an operator affiliated with $\cF$, and $\tau(|X|^p)<\infty$ for all $p\in\N$.
Let $\rxh$ be the subspace of symmetric (or self-adjoint) polynomials in $\rx$.
Obvious examples of trace-positive polynomials are sums of hermitian squares and commutators of \nc polynomials; 
however, not all trace-positive polynomials are of this form. 
To remedy this, one has to replace \nc polynomials with a certain class of regular \nc rational functions.
Let $\cK\subset\rxh$ be the convex cone of all (symmetric) \nc polynomials that can be written as
sums of hermitian squares and commutators of elements in
the $*$-subalgebra
$$\R\!\Langle x_1,\dots,x_n,(1+x_1^2)^{-1},\dots,(1+x_n^2)^{-1}\Rangle$$
of the free skew field \cite{Coh,Vol}. 
For example, the following noncommutative lift of the Motzkin polynomial
$$x_2x_1^4x_2+x_2^2x_1^2x_2^2-3x_2x_1^2x_2+1$$
belongs to $\cK$ (Example \ref{ex:motzkin}) 
even though it is not a sum of hermitian squares and commutators of \nc polynomials.
The cone $\cK$ plays a central role in our first main result, 
the solution of the unbounded tracial moment problem. 
The following theorem comprises tracial analogs of Haviland's theorem 
and Carleman's condition.

\begin{thmA}\label{t:a}
Let $\varphi:\rxh\to\R$ be a linear functional with $\varphi(1)=1$.
\begin{enumerate}[(a)]
\item There exists a tracial von Neumann algebra $(\cF,\tau)$ and $\uX\in \integ^n$ such that
$\varphi(p)=\tau(p(\uX))$ for all $p\in\rxh$
if and only if
$\varphi(\cK)=\R_{\ge0}$.

\item The equivalent conditions in (a) hold if there is $M>0$ such that
$\varphi(x_j^r)\le r! M^r$ for all $j=1,\dots,n$ and even $r\in \N$. 
\end{enumerate}
\end{thmA}

The first part of Theorem \ref{t:a} is proved as Theorem \ref{t:haviland},
while the second part is given in Theorem \ref{t:carleman}.
The proof combines methods and results from convexity \cite{Bar}, 
the theory of unbounded operators \cite{Schm0}, 
and noncommutative integration \cite{Nel,Tak2}.
Theorem \ref{t:a} is used to obtain a tracial Positivstellensatz for \nc polynomials.

\begin{thmA}\label{t:b}
The following are equivalent for $f\in\rxh$:
\begin{enumerate}[(i)]
\item $\tau(f(\uX))\ge0$ for all $(\cF,\tau)$ and $\uX\in \cF^n$;
\item $\tau(f(\uX))\ge0$ for all $(\cF,\tau)$ and $\uX\in \integ^n$;
\item $f$ lies in the closure of $\cK$ with respect to the finest locally convex topology on $\rxh$;
\item for every $\ve>0$ there exists $r\in\N$ such that
$$f+\ve\sum_{j=1}^n\sum_{k=0}^r \frac{1}{k!}x_j^{2k}$$
is a sum of hermitian squares and commutators in $\rx$.
\end{enumerate}
\end{thmA}

See Theorems \ref{t:posss} and \ref{t:lass} below for the proof. 
In addition to the aforementioned mathematics areas, 
techniques from polynomial and semidefinite optimization \cite{Las06,BKP16} are applied in the proof of Theorem \ref{t:b}.

The negative answer to Connes' embedding problem \cite{JNVWY} implies that in general, one cannot restrict (i) in Theorem \ref{t:b} to finite-dimensional von Neumann algebras (Proposition \ref{p:cec}).
Nevertheless, it suffices to consider only II\textsubscript{1} factors in Theorem \ref{t:b}(i).
On the other hand, trace polynomial inequalities that are valid in all finite von Neumann algebras 
(and are described by Theorem \ref{t:b}) do not necessarily hold in all formal tracial algebras \cite{NT}.

Alas, the cone $\cK$ is not closed in general (Proposition \ref{p:notclosed}). 
However, the statement of Theorem \ref{t:b} can be improved 
for a class of bivariate \nc polynomials.
We say that $f\in\bi$ is {\it cyclically sorted} if it is a linear combination of cyclic permutations of products of the form $x_1^ix_2^j$. It turns out (Corollary \ref{c:cyc}) that cyclically sorted \nc polynomials are trace-positive precisely when they belong to $\cK$. 
This statement is a consequence of the following new Positivstellensatz for commutative polynomials.

\begin{thmA}\label{t:c}
If $f\in\R[x,y]$ is nonnegative on $\R^2$ then there exists $k\in\N$ such that $(1+x^2)^k f$ is a sum of squares in $\R[x,y]$.
\end{thmA}

The proof (see Theorem \ref{t:bivar}) relies on real algebraic geometry of affine surfaces \cite{Sche}.
Theorem \ref{t:c} is a strenghtened solution of Hilbert's 17th problem for bivariate polynomials, 
since only rational squares with uniform univariate denominators are needed.

\subsection*{Acknowledgments} The authors thank Narutaka Ozawa for sharing his insights on matricial microstates, 
and the anonymous referee for 
their clarifying remarks on unbounded operator algebras, and 
showing us an alternative proof of Proposition \ref{t:ag}.

\section{Preliminaries}\label{sec2}

In this section we review the terminology and notation on von Neumann algebras, their affiliated operators, \nc polynomials and rational functions that are used throughout the paper.

\subsection{Affiliated and \integrable operators}

A {\it tracial von Neumann algebra} is a pair $(\cF,\tau)$ of
a finite von Neumann algebra $\cF$ with a separable predual and
a faithful normal tracial state $\tau$ on $\cF$.
Suppose $\cF$ acts on a Hilbert space $\cH$; then a closed and densely defined operator $X$ on $\cH$ is {\it affiliated with} $\cF$ if it commutes with every unitary operator in the commutant of $\cF$ in $\cB(\cH)$.

Next we review selected notions from noncommutative integration theory, following \cite{Nel,Tak2}.
For $p\in[1,\infty)$, the {\it noncommutative $L^p$-space} $L^p(\cF,\tau)$ is the completion of $\cF$ with respect to the norm $X\mapsto\tau(|X|^p)^{1/p}$ \cite[Section 3]{Nel}.
Note that $L^2(\cF,\tau)$ is a Hilbert space, and $\cF$ acts on it via the left regular representation.
Since $\tau$ is finite,
operators on $L^2(\cF,\tau)$ affiliated with $\cF$ form a $*$-algebra \cite[Section 2]{Nel}: the sum/product of closed
affiliated operators are understood as the closure of the sum/product as unbounded operators; see also \cite[Section IX.2]{Tak2}.
The sequence $(L^p(\cF,\tau))_{p\in[1,\infty)}$ can then be naturally seen as a decreasing net of nested subspaces in this algebra of operators affiliated with $\cF$ \cite[Theorem 5]{Nel}.
H\"older's inequality for noncommutative $L^p$-spaces \cite[Theorem IX.2.13(iv)]{Tak2}
states that
\begin{equation}\label{e:holder}
\big|\tau(Z_1\cdots Z_n)\big|\le \|Z_1\|_{p_1}\cdots\|Z_n\|_{p_n}
\end{equation}
for $Z_j\in L^{p_j}(\cF,\tau)$ and $\frac{1}{p_1}+\cdots+\frac{1}{p_n}=1$.
Therefore
$$\integ:=\bigcap_{p\in [1,\infty)}L^p(\cF,\tau)$$
is a $*$-algebra, and $\tau$ extends to a tracial state on $\integ$.
The algebra $\integ$ was introduced in \cite[Section 3]{Ino},
and its elements are {\it \integrable operators} affiliated with $(\cF,\tau)$.
For example,
if $\cF=L^\infty([0,1])$ and $\tau$ is the integration with respect to the Lebesgue measure on $[0,1]$,
then $\log(t) \in \integ\setminus\cF$ and $\tau(|\log(t)|^p)=p!$.

\subsection{Noncommutative polynomials and rational functions}

Let $\ux=(x_1,\dots,x_n)$ be a tuple of freely noncommuting variables, and let $\mx$ be the free monoid of words in $\ux$. Let $\rx$ be the real free $*$-algebra of {\it \nc polynomials} over $\ux$, with the involution given by $x_j^*=x_j$ for $j=1,\dots,n$.
For $d\in\N$ let $\rxd{d}\subset\rx$ be the subspace of \nc polynomials of degree at most $d$.
The universal skew field of fractions of $\rx$ is the free skew $*$-field $\rat$ (see e.g. \cite{Coh,KPV,Vol}).
In this paper, we consider the following $*$-subalgebra of $\rat$:
$$\cA:=\R\!\Langle x_1,\dots,x_n,(1+x_1^2)^{-1},\dots,(1+x_n^2)^{-1}\Rangle.$$
Alternatively, $\cA$ can be viewed as the free product of $n$ copies of $\R[t,\frac{1}{1+t^2}]$.
For $d\in\N$ let $\cA_d\subset\cA$ denote the subspace of elements that are linear combinations of products in $x_1,\dots,x_n,(1+x_1^2)^{-1},\dots,(1+x_n^2)^{-1}$ of length at most $d$.

If $\uX$ is a tuple of (possibly unbounded) self-adjoint operators affiliated with a finite von Neumann algebra $\cF$, then the evaluation
$$a(\uX)=a(X_1,\dots,X_n,(I+X_1^2)^{-1},\dots,(I+X_n^2)^{-1})$$
is well-defined for every $a\in\cA$ because the affiliated operators form an algebra and $(I+X)^{-1}$ is a bounded operator for a self-adjoint $X$.
Furthermore, if $\uX\in\integ^n$ then $\tau(a(\uX))$ is well-defined. 
More precisely, if $a\in\cA_p$ then
$\tau(a(\uX))$ is well-defined for every $\uX\in L^p(\cF,\tau)^n$ by H\"older's inequality \eqref{e:holder}.

\section{Sums of hermitian squares with denominators, and the cone \texorpdfstring{$\cK$}{K}}\label{sec3}

This section is devoted to the introduction and first properties of the convex cone $\cK$ (see \eqref{e:thecone} below) that is essential for the moment problem and positivity certificates of this paper.

Given a $*$-algebra $\cB$ let
\begin{align*}
\cB_\her &= \{b\in\cB\colon b^*=b \},\\
\sos{\cB} &=\left\{\sum_i b_ib_i^*\colon b_i\in\cB \right\}, \\
\soc{\cB} &=\spa\left\{b_1b_2-b_2b_1\colon b_1,b_2\in\cB \right\}.
\end{align*}

The following lemma lists relations between the above convex cones and subspaces
in the $*$-algebras $\rx$ and $\cA$.

\begin{lem}\label{l:cones}
Under the natural embedding $\rx\subset\cA$,
\begin{enumerate}[(a)]
\item $\soc{\cA}\cap\rx=\soc{\rx}$;
\item $\sos{\cA}\cap\rx=\sos{\rx}$;
\item $\sos{\cA}\cap\soc{\cA}=\{0\}$;
\item $(\sos{\cA}+\soc{\cA})\cap \rx\supsetneq\sos{\rx}+\soc{\rx}$ for $n\ge2$;
\item$(1+x_j^2)^{-m},1-(1+x_j^2)^{-m} \in\sos{\cA}$ for all $m\in\N$.
\end{enumerate}
\end{lem}

\begin{proof}
(a) There is an embedding of $\cA$ into formal power series $\ser$ that sends
$(1+x_j^2)^{-1}$ to $\sum_{k=0}^\infty (-1)^k x_j^{2k}$.
If $s,t\in\ser$ are given as $s=\sum_{w\in\mx}s_ww$ and $t=\sum_{w\in\mx}t_ww$, then their commutator
$$[s,t]=\sum_{i=0}^\infty \sum_{|u|+|v|=i}s_ut_v \cdot [u,v]$$
is a convergent (with respect to the adic topology of the power series) series of commutators. From here we immediately deduce that $\soc{\ser}\cap \rx=\soc{\rx}$, so (a) follows.

(b) If $\uX\in\matc{k}_\her^n$ and $s\in\sos{\cA}$, then $s(\uX)$ is a positive semidefinite matrix. Hence (b) holds by the Helton-McCullough Positivstellensatz, see \cite[Theorem 0.2]{McC} or \cite[Theorem 1.1]{Hel}.

(c) For every $\uX\in\matc{k}_\her^n$ and $c\in\soc{\cA}$ we have $\tr(c(\uX))=0$. On the other hand, if $f\in\cA\setminus\{0\}$, then there are $k\in\N$ and $\uY\in\matc{k}_\her^n$ such that $f(\uY)$ is nonzero, see e.g. \cite[Remark 6.7]{Vol}. Consequently $\tr(f(\uY)f^*(\uY))>0$. Therefore (c) holds.

(d) See Example \ref{ex:motzkin} below.

(e) Since $(1+x_j^2)^{-1}=(1+x_j^2)^{-2}(1+x_j^2)$ and $1-(1+x_j^2)^{-1}=x_j^2(1+x_j^2)^{-1}$ belong to $\sos{\cA}$, so do $(1+x_j^2)^{-m}$ and $1-(1+x_j^2)^{-m}=(1-(1+x_j^2)^{-1})\sum_{i=0}^{m-1}(1+x_j^2)^{-i}$ for every $m\in\N$.
\end{proof}

\begin{exa}\label{ex:motzkin}
Let $m=x_2x_1^4x_2+x_2^2x_1^2x_2^2-3x_2x_1^2x_2+1$. Note that $m$ is a noncommutative lift of the classical Motzkin polynomial (see \cite[Proposition 1.2.2]{Mar} or \cite[Remark 1.1.2]{Sche}), which is nonnegative on $\R^2$ but not a sum of squares of polynomials. In particular, $m\notin \sos{\rx}+\soc{\rx}$. On the other hand, by \cite[Section 4.2]{Qua} we have $m=s+c$ for $s\in\sos{\cA}$ and $c\in\soc{\cA}$ where
$$s
= (1-x_1^2x_2^2)^*(1+x_1^2)^{-1}(1-x_1^2x_2^2) 
+ x_2(1+x_1^2)^{-1}(x_1^3-x_1)^2x_2 
+(x_2^2-1)(1+x_1^2)^{-1}x_1^2(x_2^2-1)
$$
and
\begin{flalign*}
&& c=2\big[x_2,[(1+x_1^2)^{-1},x_2]\big]. && \square
\end{flalign*}
\end{exa}

Let us define the convex cone
\begin{equation}\label{e:thecone}
\cK:= \big(\sos{\cA}+\soc{\cA} \big)\cap\rxh
\end{equation}
in $\rxh$.

\begin{lem}\label{l:inv}
$\cA_\her\subseteq\sos{\cA}+\soc{\cA}+\rxh$.
\end{lem}

\begin{proof}
Consider the set of formal words in $2n$ symbols $x_1,\dots,x_n,(1+x_1^2)^{-1},\dots,(1+x_n^2)^{-1}$ that do not contain subwords $x_j(1+x_j^2)^{-1}$ or $(1+x_j^2)^{-1}x_j^2$; it maps injectively into $\cA$, and its image, denoted $\cE$, is a basis of $\cA$.
An expansion of $a\in\cA$ with respect to $\cE$ will be called {\em the normal form} of $a$.
That is, univariate sub-expressions in $a$ are written as partial fractions, with the inverses on the left for the sake of bookkeeping.

It suffices to show that $u+u^* \in \sos{\cA}+\soc{\cA}+\rxh$ for all $u\in\cE$. This is done by consecutively eliminating $(1+x_1^2)^{-1},\dots,(1+x_n^2)^{-1}$ from $u$; we demonstrate this only for $(1+x_1^2)^{-1}$, and the other $(1+x_j^2)^{-1}$ are eliminated in the same manner.

Every $v\in\cE$ can be uniquely written as
$$v=v_0(1+x_1^2)^{-m_1}v_1(1+x_1^2)^{-m_2}v_2\cdots (1+x_1^2)^{-m_\ell}v_{\ell}$$
where $m_k>0$ and $v_k\in\cE$ with $v_j\neq 1$ for $0<j<\ell$ do not contain $(1+x_1^2)^{-1}$.
Set $\delta(v)=\ell$. More generally, for $a\in\cA$ let $\delta(a)$ be the maximum of $\delta(v)$ for $v$ appearing in the normal form of $a$.

Assume that $(1+x_1^2)^{-1}$ appears in $u$ (as otherwise there is nothing to be done); that is, $\delta(u)\ge1$.
After subtracting an element of $\soc{\cA}$ from $u+u^*$ and taking a maximal term with respect to $\delta$, we can assume that $u$ starts with $(1+x_1^2)^{-1}$.
We can write $u=u'(1+x_1^2)^{-m}u''$ for $m>0$ and $u',u''\in\cE$ with $\delta(u')=\lfloor\frac{\delta(u)}{2}\rfloor$ and $\delta(u'')=\delta(u)-1-\delta(u')$.
Then
\begin{equation}\label{e:u}
\begin{split}
u+u^*
&=
\big(u'+{u''}^*\big)(1+x_1^2)^{-m}\big(u'+{u''}^*\big)^* \\
&\quad+u'(1-(1+x_1^2)^{-m}){u'}^* \\
&\quad+{u''}^*(1-(1+x_1^2)^{-m}){u''} \\
&\quad- u'{u'}^*-{u''}^*{u''}.
\end{split}
\end{equation}
The first three terms of \eqref{e:u} belong to $\sos{\cA}$ by Lemma \ref{l:cones}(e).
Furthermore, note that $\delta(u''{u''}^*)<\delta(u)$, and since $u'$ is either $1$ (if $\delta(u)=1$) or starts with $(1+x_1^2)^{-1}$,
there is $c\in\soc{\cA}$ such that
$\delta(u'{u'}^*-c)<\delta(u)$.
Using \eqref{e:u} and induction on $\delta(u)$ it then follows that
$u+u^*\in \sos{\cA}+\soc{\cA}+\rxh$.
\end{proof}

Let $\cA_\C=\C\otimes_\R\cA$ denote the complexification of $\cA$.

\begin{lem}\label{l:riesz}
Every linear functional $\varphi:\rxh\to\R$ satisfying $\varphi(\cK)=\R_{\ge0}$ extends to a linear $*$-functional $\phi:\cA_\C\to\C$ satisfying $\phi(\sos{\cA_\C}+\soc{\cA_\C})=\R_{\ge0}$.
\end{lem}

\begin{proof}
Note that $\sos{\cA}+\soc{\cA}$ is a convex cone in $\cA_\her$, and
$$\cA_\her=\sos{\cA}+\soc{\cA}\cap \cA_\her+\rxh$$
by Lemma \ref{l:inv}. 
Hence $\varphi$ extends to a linear fuctional $\phi':\cA_\her\to\R$ 
that satisfies $\phi'(\sos{\cA}+\soc{\cA}\cap\cA_\her)=\R_{\ge0}$ 
by the Riesz extension theorem \cite[Proposition 1.7]{Schm1}.
Let $\rho,\iota:\cA_\C\to\cA$ be $\R$-linear maps given by 
$\rho(a)=\frac12(a+\overline{a})$ and $\iota(a)=\frac1{2i}(a-\overline{a})$.
Thus $a=\rho(a)+i\iota(a)$ for $a\in\cA_\C$.
Let $\phi:\cA_\C\to\C$ be defined as
$$\phi(a) = \frac12\phi'(\rho(a)+\rho(a)^*)+\frac{i}{2}\phi'(\iota(a)+\iota(a)^*)$$
for $a\in\cA_\C$.
Then $\phi$ is a $*$-functional, 
$$
\phi([a,b])
=\phi\Big([\rho(a),\rho(b)]-[\iota(a),\iota(b)]
+i\big([\rho(a),\iota(b)]+[\iota(a),\rho(b)]\big)\Big)=0$$
and
\begin{align*}
\phi(aa^*)
&=\phi\Big(\rho(a)\rho(a)^*+\iota(a)\iota(a)^*+i\big(\iota(a)\rho(a)^*-\rho(a)\iota(a)^*\big)\Big)\\
&=\phi'\big(\rho(a)\rho(a)^*+\iota(a)\iota(a)^*\big)\ge0
\end{align*}
for all $a,b\in\cA_\C$.
Therefore $\phi(\sos{\cA_\C}+\soc{\cA_\C})=\R_{\ge0}$.
\end{proof}

\section{Unbounded tracial moment problem}\label{sec4}

The tracial analog of the moment problem for probability measures with compact support was solved in \cite{Had,Rad}. 
To obtain our rational Positivstellensatz on global trace positivity, one has to consider analogs of probability measures with non-compact support.
In this section we solve the unbounded tracial moment problem for \ncmeasures.

\begin{prop}\label{p:gns}
For every linear $*$-functional $\phi:\cA_\C\to\C$ satisfying $\phi(1)=1$ and $\phi(\sos{\cA_\C}+\soc{\cA_\C})=\R_{\ge0}$ there exist a tracial von Neumann algebra $(\cF,\tau)$ and $\uX\in \integ_\her^n$ such that
$\phi(a)=\tau(a(\uX))$ for all $a\in\cA_\C$.
\end{prop}

\begin{proof}
We split the proof, which is a version of the Gelfand--Naimark--Segal construction
for unbounded functionals on $\cA$ that produces power-integrable operators, 
in several steps.

\emph{Step 1: Construction of unbounded operators.}
On $\cA_\C$ we define a semi-scalar product $\langle a,b\rangle =\phi(ab^*)$. By the Cauchy-Schwarz
inequality for semi-scalar products,
$$\cN=\{a\in\cA_\C\colon \phi(aa^*)=0 \}$$
is a vector subspace of $\cA_\C$ and $\cN^*=\cN$.
Furthermore, for every $a\in\cN$, $b\in\cA_\C$ and $\ve>0$ we
have
$$0\le \phi\left((a^*\pm\varepsilon b)(a^*\pm\varepsilon b)^* \right)
=\varepsilon \left(\varepsilon \phi(bb^*) \pm 2\re \phi(ba)\right);$$
since $\varepsilon>0$ was arbitrary, and $a$ can be replaced by $ia$, it follows that $ba,ab\in\cN$. Hence $\cN$ is an ideal in $\cA_\C$. Let $\cH$ be the completion of $\cA_\C/\cN$ with respect to $\langle\cdot,\cdot\rangle$. Then $\cH$ is a separable Hilbert space; let $\vv a\in\cH$ denote the vector corresponding to $a\in\cA_\C$. The left multiplication by $x_j$ in $\cA_\C$ induces a densely defined symmetric operator $X'_j$ on $\cH$. In particular, $X_j'$ is closable by \cite[Section 3.1]{Schm0}; let $X_j$ be its closure. Since $x_j^2+1$ is invertible in $\cA$, the elements $x_j+i$ and $x_j-i$ in $\cA_\C$ are also invertible.
Hence the linear operators $X_j'+iI$ and $X_j'-iI$ are invertible on $\cA_\C/\cN$.
Therefore $X_j$ is a self-adjoint operator by \cite[Proposition 3.8]{Schm0}.
Note that
$$a(\uX)\vv{b}=\vv{ab}$$
for all $a,b\in\cA_\C$.
Then $-i$ belongs to the resolvent set of $X_j$, the resolvent $(X_j+iI)^{-1}$ is a bounded operator on $\cH$ and
\begin{equation}\label{e:ran}
\ran(X_j+iI)^{-1} = \dom X_j
\end{equation}
by \cite[Proposition 3.10]{Schm0}.

\emph{Step 2: A tracial von Neumann algebra.}
Let $\cF\subseteq \cB(\cH)$ be the von Neumann algebra generated by $(X_1+iI)^{-1},\dots,(X_n+iI)^{-1}$, i.e.,
the weak operator topology closure of the unital $*$-algebra generated by $(X_1+iI)^{-1},\dots,(X_n+iI)^{-1}$.
Define $\tau:\cF\to\C$ as $\tau(F) =\langle F\vv1,\vv1\rangle$. Then $\tau$ is a faithful normal state on $\cF$. Furthermore, $\tau$ is tracial. Indeed, let $P,Q$ be arbitrary elements of the unital $*$-algebra generated by $(X_1+iI)^{-1},\dots,(X_n+iI)^{-1}$. Then $P=p(\uX)$ and $Q=q(\uX)$ for some $p,q\in\cA_\C$.
Since $\phi(\soc{\cA_\C})=\{0\}$, we have
$$\tau(PQ)=\phi(pq)=\phi(qp)=\tau(QP).$$

By construction, the Hilbert space $L^2(\cF,\tau)$ naturally embeds into $\cH$.
In fact, $L^2(\cF,\tau)=\cH$. To see this, it suffices to show that $\vv{a}\in L^2(\cF,\tau)$ implies $\vv{x_ja}\in L^2(\cF,\tau)$ for every $a\in\cA$ and $j=1,\dots,n$.
Suppose $\vv{a}\in L^2(\cF,\tau)$; let $K\subset\cH$ be the closure of
$\{\vv{pa}\colon p\in\C[x_j,(1+x_j^2)^{-1}]\}$, and let $K_0\subseteq L^2(\cF,\tau)$ be the image of $\vv{a}$ under the $*$-algebra generated by $(X_j+iI)^{-1}$. Note that $(x_j\pm i)^{-1}\in \C[x_j,(1+x_j^2)^{-1}]$.
The map
$$\C[t,\tfrac{1}{t\pm i}]\to K,\qquad f\mapsto \vv{f(x_j)a}$$
induces a Hilbert space isomorphism
\begin{equation}\label{e:univar}
L^2(\R,\mu)\to K
\end{equation}
where $\mu$ is the finite measure on $\R$ satisfying
$$\int_\R f\,{\rm d}\mu = \phi(f(x_j)aa^*)$$
for $f\in\C[t,\tfrac{1}{t\pm i}]$.
The preimage of $K_0$ under the isomorphism \eqref{e:univar} contains the set $\{(t\pm i)^{-k}\colon k\in\N_0 \}$.
Therefore $K_0$ is dense in $K$ by \cite[Lemma 6.9]{Schm1}. Thus in particular $\vv{x_ja}\in K=\overline{K_0}\subset L^2(\cF,\tau)$, as desired.

\emph{Step 3: Affiliation.}
Let $U\in\cB(\cH)$ be a unitary in the commutant of $\cF$.
Since
\begin{equation}\label{e:com}
U(X_j+iI)^{-1}=(X_j+iI)^{-1}U
\end{equation}
and $\ran (X_j+iI)^{-1}=\dom X_j$ by \eqref{e:ran},
we have $U\dom X_j\subseteq \dom X_j$. Moreover, \eqref{e:com} then implies
$$(X_j+iI)U(X_j+iI)^{-1}(X_j+iI)=(X_j+iI)(X_j+iI)^{-1}U(X_j+iI)$$
on $\dom X_j$, from where we conclude $X_jU=UX_j$ on $\dom X_j$.
Since the unitary $U$ in the commutant of $\cF$ was arbitrary, the self-adjoint operator $X_j$ on $\cH$ is affiliated with $\cF$.

\emph{Step 4: Integrability.}
Since the positive semidefinite operator $X_j^2$ is affiliated with $\cF$,
by the spectral theorem there exists a projection valued measure $E_\lambda$ with values in $\cF$ such that
$X_j^2=\int_0^\infty \lambda\,{\rm d} E_\lambda$. Then for every $p\in\N$,
$$\tau(|X_j|^{2p})=\tau(X_j^{2p})
=\int_0^\infty \lambda^p\,{\rm d} \tau(E_\lambda)
=\int_0^\infty \lambda^p\,{\rm d} \langle E_\lambda\vv{1},\vv{1}\rangle
=\langle X_j^{2p}\vv{1},\vv{1}\rangle<\infty,
$$
where the second equality holds by \cite[Section 3]{Nel}, the third equality holds by the definition of integration, and the inequality holds since $\vv{1}\in\dom X_j^{2p}$. Therefore $X_j\in L^{2p}(\cF,\tau)$. As $p\in\N$ was arbitrary, it follows that $X_j\in \integ$.

\emph{Step 5: Conclusion.}
Finally, we have
$$\tau(a(\uX)) = \langle a(\uX)\vv{1},\vv{1}\rangle =\phi(a)$$
for every $a\in\cA_\C$.
\end{proof}

Let $(\cF,\tau)$ be a tracial von Neumann algebra and $\uX\in \integ_\her^n$. The functional $p\mapsto \tau(p(\uX))$ on $\rxh$ is called a \emph{\ncmeasure} (cf. \cite[Section 2.3]{VDN} or \cite[Section 6.4]{MS}).
We obtain the following tracial version of Haviland's theorem \cite[Theorem 3.1.2]{Mar}.

\begin{thm}\label{t:haviland}
Let $\varphi:\rxh\to\R$ be a linear functional with $\varphi(1)=1$.
Then $\varphi$ is a \ncmeasure if and only if
$\varphi(\cK)=\R_{\ge0}$.
\end{thm}

\begin{proof}
$(\Rightarrow)$ is straightforward. 
$(\Leftarrow)$ The functional $\varphi$ extends to a $*$-functional $\phi:\cA_\C\to\C$ satisfying 
$\phi(\sos{\cA_\C}+\soc{\cA_\C})$ by Lemma \ref{l:riesz}.
Then $\varphi(p)=\phi(p)=\tau(p(\uX))$ 
for some $(\cF,\tau)$ and $\uX\in\integ_\her^n$, and all $p\in\rxh$, by Proposition \ref{p:gns}.
\end{proof}

\begin{rem}
The cone $\sos{\rx}+\soc{\rx}$ is closed in $\rx$ by
\cite[Proposition 1.58 and Corollary 3.11]{BKP16} and \cite[Remark A.29]{Schm1} 
and does not contain $\cK$ by Lemma \ref{l:cones}(d).
The condition ``$\varphi\ge0$ on $\cK$'' 
is therefore more restrictive than ``$\varphi\ge0$ on $\sos{\rx}+\soc{\rx}$''.
\end{rem}

The proof of Proposition \ref{p:gns} actually implies a more general version of Theorem \ref{t:haviland}. 
Let $S\subset\rxh$. Denote
\begin{equation}\label{e:coneS}
\cK_S:= \left(\soc{\cA}+\sum_{s\in \{1\}\cup S}s\cdot\sos{\cA}\right)\cap\rxh.
\end{equation}
In particular, $\cK_\emptyset$ equals $\cK$ from \eqref{e:thecone}.
Let us say that the functional $p\mapsto \tau(p(\uX))$ for some tracial von Neumann algebra $(\cF,\tau)$ and $\uX\in\integ_\her^n$ is \emph{\ncmeasure constrained by $S$} if $s(\uX)\succeq0$ for all $s\in S$.

\begin{cor}\label{c:haviland2}
Let $\varphi:\rxh\to\R$ be a linear functional with $\varphi(1)=1$.
Then $\varphi$ is a \ncmeasure constrained by $S\subset\rxh$ if and only if
$\varphi(\cK_S)=\R_{\ge0}$.
\end{cor}

\begin{proof}
$(\Rightarrow)$ is again straightforward. 
For $(\Leftarrow)$, recall the construction of $(\cF,\tau)$ and $\uX\in\integ^n$ 
from the proof of Proposition \ref{p:gns}.
Then for every $s\in S$ and $a\in\cA_\C$,
$$\langle s(\uX)\vec{a},\vec{a}\rangle = \varphi(saa^*)\ge0.$$
Therefore $\tau(s(\uX)FF^*)\ge0$ for every $F\in\cF$, whence $s(\uX)\succeq0$.
\end{proof}

\section{A commutative intermezzo and bivariate trace-positive polynomials}\label{sec5}

In this section, 
we give a commutative analog of Theorem \ref{t:haviland}, 
and establish a refined solution of Hilbert's 17th problem for commutative bivariate polynomials
that lifts to a positivity certificate for a special class of \nc polynomials. 

Let $\uu=(t_1,\dots,t_n)$ be commuting indeterminates,
and consider the convex cone
$$\cC=\R[\uu]\cap \sos{\R\left[t_1,\dots,t_n,\tfrac{1}{1+t_1^2},\dots,\tfrac{1}{1+t_n^2}\right]}$$
in $\R[\uu]$.
We start by recording the commutative counterpart of Theorem \ref{t:haviland}.
While it follows from \cite[Theorem 13.33 and Example 13.36]{Schm1} 
and it can be proved in a similar way as Proposition \ref{p:gns} (with adaptations regarding strongly commuting operators as in \cite[Corollary 2.6]{PV}), 
we provide a simpler independent argument, inspired by real algebraic geometry.

\begin{prop}\label{p:comm}
Let $\varphi:\R[\uu]\to\R$ be a linear functional with $\varphi(1)=1$.
Then $\varphi$ comes from a probability measure if and only if $\varphi(\cC)=\R_{\ge0}$.
\end{prop}

\begin{proof}
Only the backward implication is nontrivial. Let $h=(1+t_1^2)\cdots(1+t_n^2)$, and $A=\R[\uu,\frac{1}{h}]$.
By a variant of Schm\"udgen's Positivstellensatz \cite[Corollary 3.5.2]{Mar},
every bounded nonnegative $f\in A$ lies in the closure (with respect to the finest locally convex topology)
of $\sos{A}$ in $A$. On the other hand, every $f\in A$ becomes bounded when divided by a sufficiently high power of $h$. Therefore, the closure of $\sos{A}$ agrees with the convex cone of nonnegative functions in $A$.
The rest follows by Haviland's theorem \cite[Theorem 3.1.2]{Mar}.
\end{proof}

Next, we give a new strenghened solution of Hilbert's 17th problem for bivariate polynomials.

\begin{thm}\label{t:bivar}
If $f\in\R[t_1,t_2]$ is nonnegative on $\R^2$ then there exists
$k\in\N$ such that $(1+t_1^2)^k f$ is a sum of squares in
$\R[t_1,t_2]$.
\end{thm}

\begin{proof}
We wish to show that $f$ is a sum of squares in the ring of fractions
$$A = \R[t_1,t_2]_{1+t_1^2}=\left\{(1+t_1^2)^{-k}p \colon
p\in \R[t_1,t_2],\ k\in\N_0 \right\}.$$
The ring $A$ is the coordinate ring of the affine real variety
$X=Y\times\bA^1$, where $Y$ is the projective line $\bP^1$ minus one
real point $[0:1]$ and two complex conjugate nonreal points
$[1:\pm i]$.
Let $S\subset\bA^2$ be the plane affine curve $t_1^2 + t_2^2 = 1$.
Then $S$ is isomorphic to $\bP^1$ minus two complex conjugate
points $[1:\pm i]$. So $Y\cong S\setminus\{x\}$ where $x$ is a real
point of~$S$, and hence $X\cong(S\times\bA^1)\setminus L$ is
isomorphic to $S\times\bA^1$ minus the real line $L=\{x\}\times\bA^1$.
Let $q\in\R[S]$ be such that $x$ is the only (real or complex) zero
of $q$ in $S$. Then $A=\R[X]$ is isomorphic to ring of fractions
$\R[S\times\bA^1]_q=\{q^{-k}g\colon g\in\R[S\times\bA^1]$,
$k\in\N_0\}$. To show that every nonnegative polynomial function
on $X$ is a sum of squares in $A$, it therefore suffices to see
that every nonnegative polynomial function on $S\times\bA^1$ is a
sum of squares in $\R[S \times \bA^1]$. The latter is in fact true
and was proved in \cite[Theorem 2]{SW}.
\end{proof}

\begin{rem}
Recall that every nonnegative real polynomial is a sum of squares of
rational functions (by Hilbert's 17th problem, as solved by Artin).
Speaking informally, Theorem \ref{t:bivar} states that $1+t_1^2$ is a
uniform denominator for such sums of squares in the case of two variables.
Let us point out how this relates to several previous results.
If $h\in\R[t_1,t_2]$ is such that its homogenization is a positive
definite form, then for every nonnegative $f\in\R[t_1,t_2]$ there exists
$k\in\N$ such that $h^k f$ is a sum of squares in $\R[t_1,t_2]$
(\cite[Corollary 3.12]{Sche0}). This result applies, in particular, to
$h=1+t_1^2+t_2^2$, but not to $h=1+t_1^2$. In fact, one checks easily
that Theorem \ref{t:bivar} implies the statement for $h=1+t_1^2+t_2^2$.

A weaker uniform denominator result, which however is valid for any
number of variables, says that if the homogenizations of
$f,\,h\in\R[t_1,\dots,t_n]$ are both positive definite, then there
exists $k\in\N$ such that $h^kf$ is a sum of squares of polynomials
(\cite{Rez0} for $h=1+\sum_it_i^2$ and
\cite[Remark 4.6]{Sche1} in general).
\end{rem}

Let us now apply the preceding commutative result to a special class of bivariate \nc polynomials.
Let $n=2$ and $\ux=(x_1,x_2)$.
We say that $f\in\bi$ is {\bf cyclically sorted} \cite[Definition 4.1]{KS} if
$$f\in \spa\{x_1^ix_2^j\colon i,j\in\N_0 \}+\soc{\bi}.$$
The noncommutative lift of the Motzkin polynomial in Example \ref{ex:motzkin} is cyclically sorted.

\begin{cor}\label{c:cyc}
If $f\in\bih$ is cyclically sorted, then the following are equivalent:
\begin{enumerate}[(i)]
\item $f(\underline{\xi})\ge0$ for all $\underline{\xi}\in\R^2$;
\item $\tau(f(\uX))\ge0$ for all $\uX\in\matc{k}_\her^2$ and $k\in\N$;
\item $\tau(f(\uX))\ge0$ for every tracial von Neumann algebra $(\cF,\tau)$ and $\uX\in\cF_\her^2$;
\item $f\in\cK$.
\end{enumerate}
\end{cor}

\def\cAbi{\R\!\Langle x_1,x_2,(1+x_1^2)^{-1}\Rangle}

\begin{proof}
(iv)$\Rightarrow$(iii)$\Rightarrow$(ii)$\Rightarrow$(i) is clear.

(i)$\Rightarrow$(iv) Let $\pi:\cAbi\to\R[t_1,t_2,(1+t_1^2)^{-1}]$ be the homomorphism 
given by $\pi(x_1)=t_1$ and $\pi(x_2)=t_2$.
Let
$$V=\spa\{(1+x_1^2)^{-k}x_1^ix_2^j\colon i,j,k\in\N_0 \}\subset\cAbi$$
and observe that there is a unique linear map
$\eta:\R[t_1,t_2,(1+t_1^2)^{-1}]\to V$ such that $\pi\circ\eta=\id$.
Furthermore,
\begin{align*}
\eta\Big((1+t_1^2)^{-k_1}t_1^{i_1}t_2^{j_1}\Big)
\eta\Big((1+t_1^2)^{-k_2}t_1^{i_2}t_2^{j_2}\Big)^*
&=(1+x_1^2)^{-k_1}x_1^{i_1}x_2^{j_1+j_2}x_1^{i_2}(1+x_1^2)^{-k_2} \\
&=(1+x_1^2)^{-k_1-k_2}x_1^{i_1+i_2}x_2^{j_1+j_2} \\
&\quad +\Big[
(1+x_1^2)^{-k_1}x_1^{i_1}x_2^{j_1+j_2},x_1^{i_2}(1+x_1^2)^{-k_2}
\Big]
\end{align*}
and so 
\begin{equation}\label{e:square}
\eta(a)\eta(b)^* \in V+\soc{\cA}
\end{equation}
for all $a,b\in \R[t_1,t_2,(1+t_1^2)^{-1}]$.
By (i), $\pi(f)$ is a nonnegative bivariate polynomial,
so by Theorem \ref{t:bivar} there exist $s_1,\dots,s_\ell\in \R[t_1,t_2,(1+t_1^2)^{-1}]$ such that
$\pi(f)=s_1^2+\cdots+s_\ell^2$.
Then
$$\tilde f :=\eta(s_1)\eta(s_1)^*+\cdots+\eta(s_\ell)\eta(s_\ell)^* \in\cK.$$
Since $f,\tilde{f} \in V+\soc{\cA}$ by \eqref{e:square} and $\pi(f)=\pi(\tilde{f})$, we have $f-\tilde f\in \soc{\cA}$.
Therefore $f\in\cK$.
\end{proof}

\section{Globally trace-positive polynomials}\label{sec6}

In this section we characterize multivariate trace-positive polynomials as the closure of $\cK$ in Theorem \ref{t:posss}. Closedeness and stability of the cone $\cK$ are also discussed.
Furthermore, Proposition \ref{p:cec} touches upon a connection between global trace positivity and  Connes' embedding problem. 

\subsection{A Positivstellensatz}

Solvability of the general moment problem gives rise to 
the following description of trace-positive polynomials.

\begin{thm}\label{t:posss}
For $f\in\rxh$, the following are equivalent:
\begin{enumerate}[(i)]
\item $\tau(f(\uX))\ge0$ for every tracial von Neumann algebra $(\cF,\tau)$ and $\uX\in\cF_\her^n$;
\item $\tau(f(\uX))\ge0$ for every tracial von Neumann algebra $(\cF,\tau)$ and $\uX\in\integ_\her^n$;
\item $\tau(f(\uX))\ge0$ for every tracial von Neumann algebra $(\cF,\tau)$ and
$\uX\in L^{\deg f}(\cF,\tau)_\her^n$;
\item $f$ lies in the closure of $\cK$ with respect to the finest locally convex topology on $\rxh$.
\end{enumerate}
\end{thm}

\begin{proof}

(iii)$\Rightarrow$(ii)$\Rightarrow$(i) Clear.

(i)$\Rightarrow$(iii) Let $p=\deg f$. Suppose $\tau(f(\uX))<0$ for some $\uX\in L^p(\cF,\tau)_\her^n$.
By H\"older's inequality \eqref{e:holder} and the triangle inequality, there exists a positive constant $c$ such that for every $Y\in L^p(\cF,\tau)$ satisfying $\|X_1-Y\|_p\le1$,
\begin{equation}\label{e:est}
\big|\tau(f(X_1,\dots,X_n))-\tau(f(Y,X_2,\dots,X_n)) \big|\le c\|X_1-Y\|_p.
\end{equation}
By \cite[Theorem IX.2.13(ii)]{Tak2}, $\cF$ is dense in $L^p(\cF,\tau)$ with respect to the $p$-norm.
Hence by \eqref{e:est} there exists $Y_1\in\cF$ such that $\tau(f(Y_1,X_2,\dots,X_n))<0$.
Continuing in this fashion, we can replace each $X_j$ with a bounded operator $Y_j$, and thus obtain $\uY\in\cF^n$ that satisfies $\tau(f(\uY))<0$.

(ii)$\Rightarrow$(iv) By the Hahn-Banach separation theorem \cite[Theorem 3.4]{Bar}, $f\notin\overline{\cK}$ if and only if there is a functional $\varphi:\rxh\to\R$ such that $\varphi(f)<0$ and $\varphi(\cK)=\R_{\ge0}$. Note that such a functional is nonzero, and $\varphi(p)^2\le\varphi(1)\varphi(p^2)$ for all $p\in\rxh$ implies $\varphi(1)>0$; thus we can rescale it to $\varphi(1)=1$. 
By Theorem \ref{t:haviland}, there exist $(\cF,\tau)$ and $\uX\in\integ_\her^n$ such that 
$\varphi(p)=\tau(p(\uX))$ for $p\in\rx$. Then $\tau(f(\uX))<0$.

(iv)$\Rightarrow$(ii) Every $\uX\in\integ_\her^n$ gives rise to a functional $\phi:a\mapsto \tau(a(\uX))$ on $\cA$, and $\phi(\cK)=\R_{\ge0}$.
The restriction of $\phi$ to $\rxh$ is continuous with respect to the finest locally convex topology on $\rxh$.
Hence $f\in\overline{\cK}$ implies $\phi(f)\ge0$.
\end{proof}

\begin{rem}
In Theorem \ref{t:posss}, it suffices to restrict to II\textsubscript{1} factors $\cF$ (which have unique tracial states), 
since every tracial von Neumann algebra embeds 
into a II\textsubscript{1} factor \cite[Theorem 2.5]{Dyk} (cf. \cite[Proposition 2.2]{KMV}).
\end{rem}

For $R>0$ let
$$\cM_R=\sos{\rx}+\soc{\rx}+\sum_{j=1}^n (R-x_j^2)\cdot\sos{\rx}.$$
We can also describe the closure of $\cK$ without denominators as follows
(another alternative with fewer quantifiers is given in Theorem \ref{t:lass} below).

\begin{cor}\label{c:arch}
The closure of $\cK$ in the finest locally convex topology equals
\begin{equation}\label{e:arch}
\overline{\cK}=\left\{f\in\rxh\colon f+\ve\in \cM_R  \text{ for all }\ve,R>0\right\}.
\end{equation}
\end{cor}

\begin{proof}
By Theorem \ref{t:posss}, $f\in\rxh$ belongs to $\overline{\cK}$ if and only if it has nonnegative trace on all tuples of operators from tracial von Neumann algebras.
In other words, for every $R>0$ we have
$\tau(f(\uX))\ge0$ for all $\uX\in\cF^n$ with $\|X_1\|,\dots,\|X_n\|\le\sqrt{R}$.
By \cite[Theorem 3.12]{KS}, such $f$ belongs to the right-hand side of \eqref{e:arch}.
\end{proof}

\subsection{Trace positivity on matrices}

Theorem \ref{t:posss} considers global trace positivity over all tracial von Neumann algebras,
which in principle could coincide with global trace positivity over matrices of all finite dimensions.
The recent resolution of Connes' embedding problem \cite{JNVWY}
is equivalent to the existence of:
\begin{enumerate}
\item an everywhere convergent \nc power series that has nonnegative trace on all matrix tuples, but not on a tuple of operators from a von Neumann algebra \cite[Corollary 1.2]{Rad};
\item an \nc polynomial that has nonnegative trace on all tuples of matrix contractions, but not on a tuple of contractions from a von Neumann algebra \cite[Theorem 1.6]{KS}.
\end{enumerate}
In the context of global trace positivity of \nc polynomials,
these facts inspire the following (superficially stronger) statement.

\begin{prop}\label{p:cec}
There exists $f\in\rxh$ such that
\begin{enumerate}[(i)]
\item $\tr(f(\uX))\ge0$ for all $\uX\in\matc{k}_\her^n$ and $k\in\N$;
\item $\tau(f(\uY))<0$ for some tracial von Neumann algebra $(\cF,\tau)$ and $\uY\in\cF_\her^n$.
\end{enumerate}
\end{prop}

To prove this, we first relate trace positivity on arbitrary matrix tuples
with the existence of matricial microstates in free probability \cite{Voi}.

\begin{prop}\label{p:microstates}
For a tracial von Neumann algebra $(\cF,\tau)$ and $\uY\in\cF_\her^n$, the following are equivalent:
\begin{enumerate}[(i)]
\item for every $f\in\rxh$, $\tr(f(\uX))\ge0$ for all $\uX\in\matc{k}_\her^n$ and $k\in\N$ implies
$\tau(f(\uY))\ge0$;
\item for all $\ve>0$ and $d\in\N$ there are $k\in\N$ and $\uX\in \matc{k}_\her^n$
such that
$$\left|\tau(w(\uY))-\frac1k \tr(w(\uX)) \right|<\ve$$
for all $w\in\mx_d$.
\end{enumerate}
\end{prop}

\begin{proof}
The implication (ii)$\Rightarrow$(i) clearly holds.

Now assume (i) holds, and fix $d\in\N$.
Let $L\in \rxd{d}^\vee$ be given by $L(p)=\re\tau(p(\uY))$,
and let $C\subseteq \rxd{d}^\vee$ be the closed convex hull of
$$C_0=\left\{p\mapsto \frac1k\re\tr(p(\uX))\colon k\in\N, \uX\in \matc{k}^n\right\}.$$
Suppose $L\notin C$. By the Hahn-Banach separation theorem,
there exist $f_0\in \rxd{d}\cong \rxd{d}^{\vee\vee}$ and $\gamma\in\R$ such that
$L(f_0)<\gamma< L'(f_0)$ for all $L'\in C$.
Let $f=\frac12(f+f^*)-\gamma \in\rxh$.
Then $L(f)<0< L'(f)$ for all $L'\in C_0$, which contradicts (i).

Therefore (i) implies $L\in C$, so every neighborhood of $L$ in $\rxd{d}^\vee$ contains a convex combination with rational coefficients of elements in $C_0$. By arranging finite sets of matrices into block diagonal matrices, we see that $C_0$ is closed under convex combinations with rational coefficients. Thus every neighborhood of $L$ contains an element of $C_0$, so (ii) holds.
\end{proof}

\begin{proof}[Proof of Proposition \ref{p:cec}]
Combine
Proposition \ref{p:microstates},
the equivalence of Connes' embedding problem and existence of microstates \cite[Proposition 3.3]{CD}
(also e.g. \cite[Section 7.4]{Voi} or \cite[Theorem 7]{Oza}),
and the resolution of Connes' embedding problem \cite[Section 1.3]{JNVWY}.
\end{proof}

\subsection{The cone \texorpdfstring{$\cK$}{K} is neither closed nor stable}

In this subsection we explore further properties of the cone $\cK$.
We require two auxiliary lemmas.

\begin{lem}\label{l:elem}
Let $a_0,a_1,a_2\ge0$ and $n\in\N$.
If $a_2a_0^n\ge n^n a_1^{n+1}$ then $a_2-(n+1)a_1+a_0\ge0$.
\end{lem}

\begin{proof}
If $a_0=0$, then $a_1=0$, and so $a_2-(n+1)a_1+a_0\ge0$ holds. If $a_0\neq0$, then
\begin{align*}
a_2-(n+1)a_1+a_0
& \ge  \frac{n^n a_1^{n+1}}{a_0^n}-(n+1)a_1+a_0 \\
&= \frac{a_0}{n}\left(\left(\frac{na_1}{a_0}\right)^{n+1}-(n+1)\left(\frac{na_1}{a_0}\right)+n\right) \\
&=\frac{a_0}{n}\left(\frac{na_1}{a_0}-1\right)\sum_{k=0}^{n-1}\left(\left(\frac{na_1}{a_0}\right)^k-1\right).
\end{align*}
The last two factors are both nonnegative if $na_1\ge a_0$ and both negative if $na_1<a_0$,
so $a_2-(n+1)a_1+a_0\ge0$ holds.
\end{proof}

\begin{lem}\label{l:alt}
Let $(\cF,\tau)$ be a tracial von Neumann algebra and let $a,b\in\cF_\her$. Then
$$2\tau\left((ba^2b)^{\frac{3}{2}} \right)\le \tau\left(a^4b^2+b^2a^4\right).$$
\end{lem}

\begin{proof}
By the Araki-Lieb-Thirring inequality \cite{LT}
(more precisely, its von Neumann algebra version \cite[Corollary 3]{Kos}),
$$\tau\left((ba^2b)^{\frac{3}{2}} \right)=\tau\left((|b||a|^2|b|)^{\frac{3}{2}} \right)
\le \tau\left(|b|^{\frac{3}{2}}|a|^3|b|^{\frac{3}{2}} \right)=\tau\left(|a|^3|b|^3 \right).$$
Then
\begin{align*}
\tau\left(a^4b^2+a^2b^4\right)-2\tau\left((ba^2b)^{\frac{3}{2}} \right)
&\ge\tau\left(|a|^2|b|^2|a|^2+|a||b|^4|a|\right)-2\tau\left(|a|^3|b|^3 \right)\\
&=\tau\left(
(|a|^2|b|-|a||b|^2)(|b||a|^2-|b|^2|a|)
\right)
\end{align*}
is nonnegative.
\end{proof}

A witness of $\cK\neq\overline{\cK}$ for $n\ge 3$ is a noncommutative lift of the homogenized Motzkin polynomial.

\begin{prop}\label{p:notclosed}
Let
$$h=x_2^2x_1^2x_2^2+x_1^2x_2^2x_1^2+x_3^6-3x_1x_2x_3^2x_2x_1 \in \rx.$$
Then $h\in\overline{\cK}\setminus\cK$.
\end{prop}

\begin{proof}
Consider the commutative homogenized Motzkin polynomial
$$H=t_1^4t_2^2+t_1^2t_2^4-3t_1^2t_2^2t_3^2+t_3^6\in\R[t_1,t_2,t_3]$$
which is not a sum of squares in $\R[t_1,t_2,t_3]$ by \cite[Proposition 1.2.4]{Mar}.
If $h\in\cK$, then
$(1+t_1^2)^{k_1}(1+t_2^2)^{k_2}(1+t_3^2)^{k_3}H$ is a sum of squares in $\R[t_1,t_2,t_3]$ for some $k_1,k_2,k_3\in\N$.
The lowest-degree homogeneous part of a sum of squares is again a sum of squares, so $H$ is a sum of squares in $\R[t_1,t_2,t_3]$, a contradiction. Therefore $h\notin\cK$.

Next we prove $h\in\overline{\cK}$.
Let $(\cF,\tau)$ be a tracial von Neumann algebra, and $X_1,X_2,X_3\in\cF_\her$.
Set
$$a_0=\tau(X_1^4X_2^2+X_1^2X_2^4),\quad
a_1=\tau(X_2X_1^2X_2X_3^2),\quad
a_2=\tau(X_3^6).
$$
Note that $a_0,a_1,a_2\ge0$.
By H\"older's inequality \eqref{e:holder} (with $p_1=\frac{3}{2}$ and $p_2=3$),
\begin{align*}
4\bigg(\tau(X_2X_1^2X_2X_3^2)\bigg)^3 &
\le 4\bigg(\tau\left( (X_2X_1^2X_2)^{\frac{3}{2}}\right)\bigg)^2 \tau(X_3^6)
\le \Big(\tau\left( X_1^4X_2^2+X_1^2X_2^4\right)\Big)^2 \tau(X_3^6).
\end{align*}
Therefore $4a_1^3\le a_2a_0^2$. Hence
$$\tau(h(X_1,X_2,X_3))=a_0-3a_1+a_2\ge0$$
by Lemma \ref{l:elem}.
Since $(\cF,\tau)$ and $X_1,X_2,X_3$ were arbitrary, $h\in\overline{\cK}$ by Theorem \ref{t:posss}.
\end{proof}

\begin{rem}\label{r:closed}
Consider $\cA$ with the finest locally convex topology. That is, every linear functional on $\cA$ is continuous. Then $\soc{\cA}$ is a closed subspace,
and $\sos{\cA}$ is a closed convex cone in $\cA$ by \cite[Theorem 4.5]{KPV}.
On the other hand, if $n\ge 3$ then $\sos{\cA}+\soc{\cA}$ is not closed by Proposition \ref{p:notclosed}.
\end{rem}

A desired property of a convex cone generated by (hermitian) squares is stability \cite[Section 4.1]{Mar}. Let us adapt this notion to our context. We say that $a\in\cA$ is of \emph{degree at most $d$} if $a$ can be written as an \nc polynomial of degree at most $d$ in the generators of $\cA$.
The cone $\cK$ is \emph{stable}
if there exists a function $\Delta:\N\to\N$ such that for every $d\in\N$ and $f\in\cK$ of degree at most $d$, there are $s_1,\dots,s_\ell\in\cA$ of degree at most $\Delta(d)$ such that $f-s_1s_1^*-\cdots-s_\ell s_\ell^*\in\soc{\cA}$.

\begin{lem}\label{l:bad}
If $n\ge 2$ then $\cK$ is not stable.
\end{lem}

\begin{proof}
Suppose $\cK$ is stable. Let
$M=t_1^4t_2^2+t_1^2t_2^4-3t_1^2t_2^2+1\in\R[t_1,t_2]$ be the Motzkin polynomial, and let
$m=x_2x_1^4x_2+x_2^2x_1^2x_2^2-3x_2x_1^2x_2+1\in\rxh$ be its noncommutative lift as in Example \ref{ex:motzkin}. For $\lambda\in\R$ denote
$M_\lambda(t_1,t_2)=M(\lambda t_1,\lambda t_2)\in\R[t_1,t_2]$ and
$m_\lambda(x_1,x_2)=m(\lambda x_1,\lambda x_2)\in\rxh$.
By Corollary \ref{c:cyc}, $m_\lambda\in\cK$ for all $\lambda$.
Note that $m_\lambda$ is of degree at most 6 for every $\lambda$. Thus by the stability assumption there exists $d'\in\N$ such that for every $\lambda$ there are $s_k\in\cA$ of degree at most $d'$ such that
$m_\lambda-\sum_ks_ks_k^*\in\soc{\cA}$.
Consequently $(1+t_1^2)^{d'}(1+t_2^2)^{d'}M_\lambda$ is a sum of squares of polynomials in $\R[t_1,t_2]$ for every $\lambda$. But this is impossible by the proof of \cite[Theorem 1]{Rez}. Indeed, if $\lambda\neq0$ and the polynomial $(1+t_1^2)^{d'}(1+t_2^2)^{d'}M_\lambda$ of degree $4d'+6$ is a sum of squares, then so is
$(1+\frac{1}{\lambda^2}t_1^2)^{d'}(1+\frac{1}{\lambda^2}t_2^2)^{d'}M$. Since the cone of sums of squares of degree at most $4d'+6$ is closed \cite[Corollary 3.34]{Lau}, taking the limit $\lambda\to\infty$ implies that $M$ is a sum of squares in $\R[t_1,t_2]$, a contradiction.
\end{proof}

\begin{rem}\label{r:seq}
It is unclear whether $\overline{\cK}$ coincides with the sequential closure $\cK^\ddagger$ of $\cK$ \cite{CMN},
$$\cK^\ddagger =
\{f\in\rxh:\text{there is } g\in\rxh \text{ such that }f+\ve g\in\cK \text{ for every }\ve>0\}.$$
See Theorem \ref{t:lass} below for a description of $\overline{\cK}$ resembling $\cK^\dagger$.
\end{rem}

\begin{rem}
Membership in $\cK$ can be certified with a sequence of semidefinite programs. 
Namely, $f\in\rxh$ belongs to $\cK$ if and only if 
\begin{equation}\label{e:sdp1}
f=\sum_i s_is_i^*+\sum_j [a_j,b_j],\qquad s_j,a_j,b_j\in\cA_d
\end{equation}
for some $d\in\N$, and \eqref{e:sdp1} can be rephrased as a feasibility semidefinite program.
Similarly, for $d\in\N$ let $\mu_d$ be the solution of the optimization problem
\begin{equation}
\label{e:sdp2}
\begin{aligned}
\inf_{\substack{L : \cA_{2d} \to \R \\ L \emph{ linear}}} \quad  & L(f)  \\	
\emph{s.t.}
\quad & L(1)=1 \,, \\
\quad & L(ab) = L(ba) \,, \quad \text{for all } a,b\in\cA_d\,, \\
\quad & L(s^*s)\geq0\,,\quad \text{for all } s\in\cA_{2d}\,.
\end{aligned}
\end{equation}
Then \eqref{e:sdp2} is a semidefinite program, 
and $(\mu_d)_d$ is an increasing sequence of lower bounds for $\inf_{(\cF,\tau), \uX\in\cF^n} \tau(f(\uX))$.
\end{rem}

\subsection{Tracial arithmetic-geometric mean inequality}

The following tracial version of the renowned arithmetic-geometric mean inequality is essentially known, and can be deduced from the generalized H\"older's inequality for unitarily invariant norms \cite[Exercise IV.2.7]{Bha} (cf. \cite[Theorem 4.2]{FK}) applied to the nuclear norm on a tracial von Neumann algebra.
We present an alternative argument inspired by the proof of Proposition \ref{p:notclosed}.

\begin{prop}\label{t:ag}
Let $(\cF,\tau)$ be a tracial von Neumann algebra. Then
\begin{equation}\label{e:ag}
\frac{\tau(f_1+\cdots+f_n)}{n} \ge
\tau\left(
\left(\left(f_1^\frac12\cdots f_n^\frac12\right)^*\left(f_1^\frac12\cdots f_n^\frac12\right)\right)^{\frac1n}
\right)
\end{equation}
for all positive semidefinite $f_1,\dots,f_n\in\cF$.
\end{prop}

\begin{proof}
We prove \eqref{e:ag} by induction on $n$. If $n=1$, \eqref{e:ag} is an equality.
Now assume \eqref{e:ag} holds for $n$.
Let $f_1,\dots,f_{n+1}\succeq0$,
and set $f=f_n^\frac12 \cdots f_1^\frac12 f_1^\frac12 \cdots f_n^\frac12$.
By the Araki-Lieb-Thirring inequality \cite[Corollary 3]{Kos}
and H\"older's inequality \eqref{e:holder} with $p_1=n+1$ and $p_2=\frac{n+1}{n}$,
\begin{align*}
\tau\left(\left(f_{n+1}^\frac12 f f_{n+1}^\frac12\right)^\frac{1}{n+1}\right)
&\le\tau\left(f_{n+1}^\frac{1}{2(n+1)}f^\frac{1}{n+1}f_{n+1}^\frac{1}{2(n+1)}\right) \\
&=\tau\left(f_{n+1}^\frac{1}{n+1}f^\frac{1}{n+1}\right) \\
&\le \tau\left(f_{n+1}\right)^\frac{1}{n+1}\cdot
\tau\left(f^\frac{1}{n}\right)^\frac{n}{n+1}.
\end{align*}
Therefore
\begin{equation}\label{e:ag1}
n^n\tau\left(\left(f_{n+1}^\frac12 f f_{n+1}^\frac12\right)^\frac{1}{n+1}\right)^{n+1}
\le \tau\left(f_{n+1}\right)\cdot
\left(n\tau\left(f^\frac1n\right)\right)^n.
\end{equation}
By the induction hypothesis,
\begin{equation}\label{e:ag2}
n\,\tau\left(f^\frac1n\right)
\le \tau\left(f_1+\cdots+f_n\right).
\end{equation}
Let
\begin{align*}
a_0&=\tau\left(f_1+\cdots+f_n\right), \\
a_1&=\tau\left(\left(f_{n+1}^\frac12 \cdots f_1^\frac12
f_1^\frac12 \cdots f_{n+1}^\frac12\right)^\frac{1}{n+1}\right), \\
a_2&=\tau\left(f_{n+1}\right).
\end{align*}
Then \eqref{e:ag1} and \eqref{e:ag2} imply $n^n a_1^{n+1}\le a_2a_0^n$,
so $a_2+a_0-(n+1)a_1\ge0$ by Lemma \ref{l:elem}. Therefore \eqref{e:ag} holds for $n+1$.
\end{proof}

\begin{rem}
Other (weaker) inequalities resembling \eqref{e:ag} are
\begin{equation}\label{e:ag3}
\frac{\tau(f_1+\cdots+f_n)}{n} \ge
\left|\tau\left(
f_1^\frac1n\cdots f_n^\frac1n\right)\right|,
\end{equation}
which holds by \eqref{e:holder} and the classical algebraic-geometric mean inequality, and
\begin{equation}\label{e:ag4}
\frac{\sqrt{\tau(f_1^2)}+\cdots+\sqrt{\tau(f_n^2)})}{n} \ge
\sqrt{
	\tau\left(
	\left(\left(f_1^\frac12\cdots f_n^\frac12\right)^*\left(f_1^\frac12\cdots f_n^\frac12\right)\right)^{\frac2n}
\right)
}
\end{equation}
which follows by \cite[Exercise IV.2.7]{Bha} applied to the Hilbert-Schmidt norm on $(\cF,\tau)$.
\end{rem}

\section{Denominator-free characterization of global trace positivity}\label{sec7}

While Section \ref{sec6} gives an algebraic certificate for trace positivity in terms of sums of hermitian squares with denominators, this section presents an alternative that involves only \nc polynomials. First, we give a sufficient condition for solvability of the unbounded tracial moment problem (Theorem \ref{t:carleman}). Next, we show that every trace-positive \nc polynomial can be perturbed to a sum of hermitian squares and commutators of \nc polynomials (Theorem \ref{t:lass}). Finally, we demonstrate this principle explicitly on (a noncommutative lift of) the Motzkin polynomial.

\subsection{Tracial Carleman's condition}

In this subsection we show that a variant of Carleman's condition for the Hamburger moment problem \cite[Corollary 4.10]{Schm1} is a sufficient condition for a functional to be a \ncmeasure. The following is a tracial version of Nussbaum's theorem \cite[Theorem 14.19]{Schm1}.

\begin{thm}\label{t:carleman}
Let $\varphi:\rxh\to\R$ be a linear functional satisfying $\varphi(1)=1$ and
$\varphi(pp^*)=\varphi(p^*p)\ge0$ for all $p\in\rx$.
If there is $M>0$ such that
\begin{equation}\label{e:carleman}
\varphi(x_j^{2r})\le (2r)! M^r \qquad\text{for all } r\in\N \text{ and } j=1,\dots,n,
\end{equation}
then $\varphi$ is a \ncmeasure.
\end{thm}

\begin{proof}
We extend $\varphi$ to a $*$-functional $\phi:\cx\to\C$ as
$$\phi(p)=\frac14\varphi(p+p^*+\overline{p}+\overline{p}^*)
+\frac{i}{4}\varphi(-ip+ip^*+i\overline{p}-i\overline{p}^*)$$
for $p\in\rx$. As in the proof of Lemma \ref{l:riesz} we see that
$\phi(pp^*)=\phi(p^*p)\ge0$ for $p\in\cx$, and then
$$\phi([p,q]) 
=\frac12 \phi([p,q]+[q^*,p^*])
=\frac12 \phi\Big((p-q^*)^*(p-q^*)-(p-q^*)(p-q^*)^* \Big)
=0$$
for $p,q\in\cx$.
Therefore $\phi(\sos{\cx})=\R_{\ge0}$ and $\phi(\soc{\cx})=\{0\}$.
We proceed with another variation of the Gelfand--Naimark--Segal construction following the same steps
as in the proof of Proposition \ref{p:gns}.

\def\wt{\widetilde}

\emph{Step 1: Construction of unbounded operators.}
On $\cx$, there is a semi-scalar product $\langle p,q\rangle =\phi(pq^*)$.
As in Step 1 of the proof of Proposition \ref{p:gns}, it gives rise to a separable Hilbert space $\cH$.
Let $\cD=\{\vv{p}\in\cH\colon p\in\cx\}$, which is a dense subspace of $\cH$.
For $j=1,\dots,n$ let $X_j'$ be symmetric operators on $\cD$ induced by left multiplication by $x_j$ on $\cx$.
Let $\wt{X}_j$ be the closure of $X_j'$.
Note that $\cD\subseteq \dom w(\wt{\uX})$ for every $w\in\mx$.

Set $R=\frac{1}{5\sqrt{M}}>0$, and fix $1\le j\le n$.
Let $p\in\cx$ and $t\in(-R,R)$ be arbitrary.
By the triangle inequality and the Cauchy--Schwarz inequality,
\begin{equation}\label{e:tail}
\begin{split}
\left\|
\sum_{k=m}^\infty \frac{(it)^k}{k!}\wt{X}_j^k\vv{p}
\right\|
&\le \sum_{k=m}^\infty \frac{|t|^k}{k!}\left\|\wt{X}_j^k\vv{p}\right\| \\
&= \sum_{k=m}^\infty \frac{|t|^k}{k!}\sqrt{\phi(x_j^{2k}pp^*)} \\
&\le \sum_{k=m}^\infty \frac{|t|^k}{k!}\sqrt[4]{\phi(x_j^{4k})}\sqrt[4]{\phi((pp^*)^2)} \\
&\le\sqrt[4]{\phi((pp^*)^2)} \sum_{k=m}^\infty (\sqrt{M}|t|)^k\frac{\sqrt[4]{(4k)!}}{k!}
\end{split}
\end{equation}
is arbitrary small for large $m\in\N$ since the series
$$\sum_{k=0}^\infty \frac{\sqrt[4]{(4k)!}}{k!\cdot 5^k}$$
is absolutely convergent by the ratio test.
Therefore for every $t\in(-R,R)$ there is a well-defined linear map $\wt{U}_j(t):\cD\to\cH$,
$$\wt{U}_j(t)=\exp\left(it\wt{X}_j\right)=\sum_{k=0}^\infty \frac{(it)^k}{k!}\wt{X}_j^k.$$
By the properties of the exponential function we see that $\wt{U}_j(t)$ is an isometry, and thus uniquely extends to an isometry $U_j(t):\cH\to\cH$. Furthermore, $U_j(t)^*=U_j(-t)$ and $U_j(-t)U_j(t)=U_j(t)U_j(-t)=I$ on $\cD$, and therefore on $\cH$, so $U_j(t)$ is a unitary. Furthermore, $U_j(s)U_j(t)=U_j(s+t)$ for all $s,t,s+t\in (-R,R)$.
Therefore we can extend the family $\{ U_j(t):t\in (-R,R)\}$ to
a well-defined one-parametric unitary group $\{ U_j(t):t\in \R\}$,
uniquely  determined by
\begin{equation}\label{e:extend}
U_j\left(ma\right)=U_j(a)^m
\end{equation}
for all $a\in(-R,R)$ and $m\in\Z$.
Furthermore, this one-parametric group is strongly continuous by \eqref{e:tail} and \eqref{e:extend}.
By Stone's theorem \cite[Theorem 6.2]{Schm0} there is a unique self-adjoint operator $X_j$ on $\cH$ such that
$U_j(t)=\exp(itX_j)$ for all $t\in\R$, and
\begin{equation}\label{e:der}
\begin{split}
\dom X_j&=\left\{v\in\cH\colon \lim_{\ve\to0}\frac{1}{\ve}\big(U_j(\ve)v-v\big) \text{ exists}\right\},\\
X_jv&=\lim_{\ve\to0}\frac{-i}{\ve}\big(U_j(\ve)v-v\big)\qquad \text{for all }v\in\dom X_j
\end{split}
\end{equation}
by \cite[Proposition 5.1]{Schm0}.
In particular, for $\vv{p}\in\cD$ we have $X_j\vv{p} = \wt{X}_j\vv{p}$ by the definition of $\wt{U}_j(t)$ and $U_j(t)$,
so $X_j$ is an extension of $\wt{X}_j$.

\emph{Step 2: A tracial von Neumann algebra.}
Let $\cF\subseteq \cB(\cH)$ be the von Neumann algebra generated by $U_j(t)$ for $j=1,\dots,n$ and
$t\in \R$.
Define $\tau:\cF\to\C$ as $\tau(F) =\langle F\vv1,\vv1\rangle$.
Note that
$$\langle p(\underline{\wt{X}})q(\underline{\wt{X}})\vv1,\vv1\rangle
=\phi(pq)=\phi(qp)=
\langle q(\underline{\wt{X}})p(\underline{\wt{X}})\vv1,\vv1\rangle$$
for all $p,q\in\rx$. By estimating as in \eqref{e:tail} we see that
\begin{align*}
\langle w_1(\underline{\wt{U}})w_2(\underline{\wt{U}})\vv1,\vv1\rangle
&=\lim_{m\to\infty}
\left\langle
w_1\left(\sum_{k=0}^m \frac{(it)^k}{k!}\underline{\wt{X}}^k\right)
w_2\left(\sum_{k=0}^m \frac{(it)^k}{k!}\underline{\wt{X}}^k\right)
\vv1,\vv1\right\rangle \\
&=\lim_{m\to\infty}
\left\langle
w_2\left(\sum_{k=0}^m \frac{(it)^k}{k!}\underline{\wt{X}}^k\right)
w_1\left(\sum_{k=0}^m \frac{(it)^k}{k!}\underline{\wt{X}}^k\right)
\vv1,\vv1\right\rangle \\
&=\langle w_2(\underline{\wt{U}})w_1(\underline{\wt{U}})\vv1,\vv1\rangle
\end{align*}
for all $w_1,w_2\in\mx$.
Therefore
$$\tau(w_1(\underline{U})w_2(\underline{U}))
=\tau(w_2(\underline{U})w_1(\underline{U}))$$
for all $w_1,w_2\in\mx$,
so $\tau$ is a faithful normal tracial state on $\cF$.
By construction, the Hilbert space $L^2(\cF,\tau)$ naturally embeds into $\cH$.
Since $\cD=\spa\{w(\uX)\vv{1} \colon w\in\mx\}$,
we have $\cD\subset L^2(\cF,\tau)$ by \eqref{e:der}. Therefore $\cH=L^2(\cF,\tau)$.

\emph{Step 3: Affiliation.}
Let $V\in\cB(\cH)$ be a unitary in the commutant of $\cF$, and fix $1\le j\le n$. Then $V$ commutes with all $U_j(t)$. If $v\in\dom X_j$, then $\frac{1}{\ve}(U_j(\ve)v-v)$ converges as $\ve\to 0$.
Then the same holds for
$$V\cdot \frac{1}{\ve}(U_j(\ve)v-v)=\frac{1}{\ve}(U_j(\ve)Vv-Vv),$$
so $Vv\in\dom X_j$, and $VX_jv=X_jVv$ by \eqref{e:der}. Thus $V$ commutes with $X_j$.
Consequently the self-adjoint operators $X_1,\dots,X_n$ are affiliated with $\cF$.

\emph{Step 4: Integrability.}
Follows exactly as in Step 4 of the proof of Proposition \ref{p:gns}.

\emph{Step 5: Conclusion.}
By construction we have
$$\tau(p(\uX)) = \langle p(\uX)\vv{1},\vv{1}\rangle =\phi(p)$$
for every $p\in\cx$.
\end{proof}

\subsection{Approximation with sums of hermitian squares and commutators}

The aim of this subsection is to establish the tracial version
of Lasserre's perturbation result \cite{Las06,LN}
for globally positive polynomials.

\def\eps{\ve}

\begin{thm}\label{t:lass}
For $f\in\rxh$, the following are equivalent:
\begin{enumerate}[\rm (i)]
\item $\tau(f(\uX))\ge0$ for every tracial von Neumann algebra $(\cF,\tau)$ and $\uX\in\cF_\her^n$;
\item for each $\eps>0$ there exists $r\in\N$ such that
\[
f+ \eps \sum_{j=1}^n\sum_{k=0}^r \frac{1}{k!}x_j^{2k} \in \sos{\rx}+\soc{\rx}.
\]
\end{enumerate}
\end{thm}

\def\csim{\overset{\text{cyc}^*}{\sim}}

For $r\in \N$ denote
$$\Omega_r = \sum_{j=1}^n\sum_{k=0}^r \frac{1}{k!}x_j^{2k}.$$
Note that $\Omega_r-n\in\sos{\rx}$.
Two optimization problems will be key in the proof of Theorem \ref{t:lass}.
Let $r\in\N$ and $M\in\R_{>0}$, and consider
\begin{equation}
\label{eq:Q_rm}
Q_{r,M}: \begin{cases}
\begin{aligned}
\inf_{\substack{L : \rxd{2 r} \to \R \\ L \emph{ linear, self-adjoint, tracial}}} \quad  & L(f)  \\	
\emph{s.t.}
\quad & L(M-\Omega_r)\geq0
 \,, \\
\quad & L(1) = 1\,, \\
\quad & L(p^*p)\geq0\quad \text{for all } p\in\rxd{r}\,;
\end{aligned}
\end{cases}
\end{equation}
\begin{equation}
\label{eq:Q*_rm}
Q^\vee_{r,M}: \begin{cases}
\begin{aligned}
\sup_{z\in\R} \quad  & z  \\	
\emph{s.t.}
\quad & f-z \in \Theta_{2r}^2 + \R_{\geq0} (M-\Omega_r)
\,.
\end{aligned}
\end{cases}
\end{equation}

Here $\Theta_{2r}^2\subset\rxh$ is the set of all (symmetric) \nc polynomials of degree $\leq2r$ that are cyclically equivalent to sums of (degree $\leq r$) squares.
Recall \cite{KS}: two words $u,v\in\mx$ are called cyclically equivalent
($u  \csim v$)
if $v$ or $v^*$ can be obtained from $u$ by cyclically rotating the letters in $u$.
For notational convenience, let $\cC_{r,M}=\Theta_{2r}^2 + \R_{\geq0} (M-\Omega_r)$.

\begin{lem}\label{lem:2sdp}
The optimization problems \eqref{eq:Q_rm} and \eqref{eq:Q*_rm}
are semidefinite programs dual to each other.
\end{lem}

\def\bfM{\mathbf M}
\def\bfG{\mathbf G}
\def\bfV{\mathbb V}
\begin{proof}
This is a variation on what is now standard material, cf.~\cite{BKP16}.
Encode the tracial linear functional $L:\rxd{2 r} \to \R $
with its Hankel matrix
$\bfM(L)_{u,v}=L(u^*v)$ for $u,v\in\mx_r$.
With this \eqref{eq:Q_rm} can be rewritten as
\begin{equation}
\label{eq:Q_rm2}
Q_{r,M}: \begin{cases}
\begin{aligned}
\inf
\quad
& \langle \bfM(L),\,\bfG(f) \rangle  \\	
\emph{s.t.}
\quad & \Big\langle \bfM(L),\, \bfG(M-\Omega_r)\Big\rangle\geq0
 \,, \\
 \quad & \bfM(L)_{u,v}=\bfM(L)_{\hat u, \hat v} \quad
 \text{if } u^*v \csim \hat u^* \hat v\,, \\
\quad & \bfM(L)_{1,1} = 1 \,, \\
\quad & \bfM(L)\succeq0\,.
\end{aligned}
\end{cases}
\end{equation}
where $\bfG(f)$ denote a Gram matrix of $f$, i.e.,
$f=\bfV_r^* \bfG(f)\bfV_r$ if $\bfV_r$ denote the vector
of all words in $\ux$ of degree $\leq r$.
Now \eqref{eq:Q_rm2} is easily seen to be a SDP:
the objective function is linear in the entries of $\bfM(L)$,
the first constraint is a linear inequality on the entries of $\bfM(L)$,
the equality constraints give rise to a finite set of linear equations on the entries of $\bfM(L)$, and the last constraint is a positivity constraint on $\bfM(L)$.

To recognize that \eqref{eq:Q*_rm} is an SDP,
observe that $f-z\in \cC_{r,M}$
if and only if there is a Gram matrix $\bfG(f)$, a positive semidefinite
$G\succeq0$ and $\lambda\geq0$ such that
\begin{equation}\label{eq:Q*1}
\bfV_r^* \bfG(f)\bfV_r -z \csim \bfV_r^* G \bfV_r + \lambda(M-\Omega_r).
\end{equation}
Clearly, \eqref{eq:Q*1} yields linear constraints on the entries of $G$, so maximizing $z$ over the set of feasible $G$ is a semidefinite program.

We shall now use a standard Lagrange duality approach to
show the SDPs \eqref{eq:Q_rm} and \eqref{eq:Q*_rm} are dual to each other:
\[
\begin{split}
\sup_{f-z\in \cC_{r,M}}
z & = \sup_z \inf_{L\in\cC_{r,M}^\vee} (z+ L(f-z)) \\
& \leq \inf_{L\in\cC_{r,M}^\vee} \sup_{z} (z+ L(f-z)) \\
& = \inf_{L\in\cC_{r,M}^\vee} \big( L(f) + \sup_z z(1-L(1)) \big) \\
& = \inf \big\{ L(f) \mid L\in\cC_{r,M}^\vee,\, L(1)=1\big\}.
\end{split}
\]
The first equality  comes from the fact that the inner minimization problem gives minimal value 0
if and only if $f - z \in\cC_{r,M}$.
The inequality in this chain is obvious.
The inner maximization problem in the next to last line is bounded with maximum value $0$ if and only $L(1) = 1$. Finally, the optimization problem on the last line is equivalent to \eqref{eq:Q_rm}.
\end{proof}

\begin{lem}\label{lem:closed}
The convex cone $\cC_{r,M}$ is closed in the finite dimensional
Euclidean space $\rxd{2r}$.
\end{lem}

\begin{proof}
The cone $\Theta_{2r}^2$ is well-known to be closed \cite[Proposition 1.58]{BKP16}. Hence the conclusion follows from
\cite[Theorem 3.2]{Beu07}.
\end{proof}

\begin{lem}\label{l:strongdual}
Strong duality holds for the pair
of SDPs \eqref{eq:Q_rm} and \eqref{eq:Q*_rm}.
\end{lem}

\begin{proof}
Let $\inf Q_{r,M}$ denote the optimal value of \eqref{eq:Q_rm}
and let $\sup Q^\vee_{r,M}$ denote the optimal value of \eqref{eq:Q*_rm}.
By Lemma \ref{lem:2sdp} and weak duality from semidefinite programming \cite[Theorem IV.6.2]{Bar},
$\sup Q^\vee_{r,M}\leq\inf Q_{r,M}$.
If $M<n$, then \eqref{eq:Q_rm} is not feasible, and $\sup Q^\vee_{r,M}=\infty$ because $\cC_{r,M}=\rxd{2r}$.
Hence let $M\ge n$. Then $\bfM(L)=E_{11}$ is clearly feasible for \eqref{eq:Q_rm},
whence $\inf Q_{r,M}<\infty$.

Suppose that \eqref{eq:Q*_rm} is feasible, $-\infty<\sup Q^\vee_{r,M}\leq\inf Q_{r,M}$. Note that $L(f-\inf Q_{r,M})\geq0$ for all $L\in\cC_{r,M}^\vee$. This implies that $f-\inf Q_{r,M}$ is in $\cC_{r,M}^{\vee\vee}=\cC_{r,M}$ since $\cC_{r,M}$ is closed by Lemma \ref{lem:closed}. Hence
$\sup Q^\vee_{r,M}\geq\inf Q_{r,M}$.

Finally, suppose that \eqref{eq:Q*_rm} is infeasible. Then for every
$\lambda\in\R$, $f-\lambda\not\in\cC_{r,M}$. By the
Hahn-Banach separation theorem \cite[Theorem III.1.3]{Bar}, there exists
a linear functional
$L\in\cC_{r,M}^\vee$ with $L(\cC_{r,M})\subseteq\R_{\geq0}$,
$L(1)=1$ and $L(f)<\lambda$. As $\lambda$ was arbitrary, this shows
\eqref{eq:Q_rm} is unbounded, establishing strong duality.
\end{proof}

\begin{lem}\label{l:upper}
Suppose $f\in\rx$ is uniformly bounded below on tracial von Neumann algebras,
in the sense that
$f_\star:=\inf_{(\cF,\tau), \uX\in\cF^n} \tau(f(\uX)) > -\infty$.
Then \eqref{eq:Q_rm} is feasible for $2r\geq \deg f$ and $M\ge n$,
and $\inf Q_{r,M} \nearrow f_M$ as $r\to\infty$
for some $f_M\geq f_\star$.
\end{lem}

\begin{proof}
Feasibility of \eqref{eq:Q_rm} for $2r\ge \deg f$ and $M\ge n$ is clear
(e.g., $L(p)=p(0)$ for $p\in\rx_{2r}$).
If $L$ is feasible for $Q_{r,M}$ then its
restriction is feasible for $Q_{r',M}$ for $r'<r$.
Hence the sequence $(\inf Q_{r,M})_{r\geq d}$ is increasing.

Let $L$ be feasible for $Q_{r,M}$. Observe that for $k\le r$,
the values of $L(x_j^{2k})$ are bounded by the linear inequality,
\begin{equation}\label{e:QrMbd}
L(x_j^{2k})\leq k!M.
\end{equation}
Then Hadwin's noncommutative H\"older inequality for linear functionals on the free algebra (see \cite[Proof of Theorem 1.3]{Had}) implies a bound
\[
|L(x_{i_1}\cdots x_{i_\delta})|\leq
\prod_{j=1}^\delta \sqrt[2^\delta]{L(x_{i_j}^{2^\delta})}
\leq
\prod_{j=1}^\delta \sqrt[2^\delta] {(2^{\delta-1})!M} =: c_{\delta}
\]
for all $2^\delta\leq 2r$.
In particular, if $r\ge 2^{\deg f-1}$ then $L(f)\le s\cdot c_{\deg f}$,
where $s$ is the number of summands in $f$.

Hence $(\inf Q_{r,M})_r$ is an increasing function bounded from above,
whence $\inf Q_{r,M} \nearrow f_M$ as $r\to\infty$, for some $f_M$.
It remains to show $f_M\geq f_\star$.

To each $L:\rx\to\R$ we assign the infinite Hankel matrix
$\bfM(L)$ as in the proof of Lemma \ref{lem:2sdp}.
If $L$ acts only on $\rxd{2r}$ we extend it by $0$ to all of $\rx$.
We also scale each $L:\rx\to\R$ to $\check L:\rx\to\R$ by
$\check L(w) = \frac1{c_{|w|}}L(w)$
for a word $w\in\mx$.

Let $L^{(r)}$ be an optimizer of \eqref{eq:Q_rm}, and consider the sequence
$
(\bfM(\check L^{(r)}))_{r\in\N}.
$
Each entry in each infinite matrix is bounded by $1$ in absolute value,
so we may consider this a sequence in the unit ball $B_1$ of $\ell^\infty$. By the Banach-Alaoglu theorem \cite[Theorem III.2.9]{Bar}, $B_1$ is compact in the weak-$*$ topology of $\ell^\infty$. Hence there
is
an element $\bfM=\bfM(\check L)$ in $B_1$ and a subsequence $(\bfM(\check L^{(r_k)}))_{k\in\N}$ converging to $\bfM(\check L)$. In particular,
$\check L^{(r_k)}(w)\to \check L(w)$ as $k\to\infty$, for all $w\in\mx$.
Now define
\[
\hat L:\rx\to\R,\quad w\mapsto c_{|w|} \check L(w).
\]
Then $L^{(r_k)}|_{\rxd{\delta}}\to \hat L|_{\rxd{\delta}}$ as $k\to \infty$, for every $\delta\in\N$.
In particular, $\hat L$ is a $*$-functional, 
$\hat L(1)=1$, $\hat L(f)=f_M$, 
$\hat L(\soc{\rx})=\{0\}$ and $\hat L(\sos{\rx})=\R_{\ge0}$.

Let $m\in\N$ be arbitrary. Then for every $r_k\ge m$,
$$L^{(r_k)}(x_j^{2m})\le m! M$$
by \eqref{e:QrMbd}.  Consequently
$$\hat L(x_j^{2m})\le m! M.$$
Therefore $L$ is a \ncmeasure by Theorem \ref{t:carleman}. 
In particular, $f_M\ge f_\star$ since 
$f_\star=\inf_{(\cF,\tau), \uX\in\integ^n} \tau(f(\uX))$ by (i)$\Leftrightarrow$(ii) of Theorem \ref{t:posss}.
\end{proof}

\begin{proof}[Proof of Theorem \ref{t:lass}]
(ii)$\Rightarrow$(i) Let $(\cF,\tau)$ and $\uX\in\cF_\her^n$ be arbitrary. Then for every $r\in\N$,
$$\tau\left(\Omega_r(\uX)\right)
\le \sum_{j=1}^n\tau\left(\exp(X_j^2)\right)=:M<\infty.$$
By (ii), for every $\eps>0$ there exists $r\in\N$ such that
$$\tau\left(f(\uX)+ \ve \Omega_r(\uX)\right)\ge0.$$
Therefore for every $\ve>0$,
$$\tau(f(\uX))\ge -\ve M,$$
and so $\tau(f(\uX))\ge0$.

(i)$\Rightarrow$(ii)
Denote $f_\star=\inf_{(\cF,\tau), \uX} \tau(f(\uX))$.

First assume $f_\star>0$. Let $M > \max\{\frac{1}{f_\star},n\}$ be arbitrary.
By Lemmas \ref{l:strongdual} and \ref{l:upper} there exists $r_M>0$ such that
$\sup Q^\vee_{r_M,M}> f_\star-\frac1M$.
That is, there are $z_M\ge f_\star-\frac1M$, $\lambda_M\ge0$ and $q_M\in \Theta_{2r_M}^2$ such that
\begin{equation}\label{e:sdpsol}
f-z_M=q_M+\lambda_M\left(M-\Omega_{r_M}\right).
\end{equation}
Evaluating \eqref{e:sdpsol} at $\uX=0\in\R^n$ gives
$$
f(0)-f_\star+\frac1M
\ge f(0)-z_M
= q_M(0)+ \lambda_M (M-\Omega_{r_M}(0))
\ge\lambda_M (M-n),
$$
and therefore
\begin{equation}\label{e:lamest}
\lambda_M\le \frac{f(0)-f_\star+\frac1M}{M-n}.
\end{equation}
The right-hand side of \eqref{e:lamest} goes to $0$ as $M\to\infty$.
By \eqref{e:sdpsol},
$$f+\lambda_M\Omega_{r_M}=z_M+q_M+\lambda_M M \in \Theta_{2r_M}^2,$$
and $\lambda_M\to0$ as $M\to \infty$.
Therefore (ii) holds.

Now assume $f_\star=0$, and let $\ve>0$ be arbitrary.
By applying (i)$\Rightarrow$(ii) to the \nc polynomial $f+\frac{n\ve}{2}$ and $\frac{\ve}{2}>0$,
there exists $r\in\N$ such that
$\left(f+\frac{n\ve}{2}\right)+\frac{\ve}{2} \Omega_r\in \Theta_{2r}^2$.
But this \nc polynomial equals $f+\ve \Omega_r -\frac{\ve}{2} (\Omega_r-n)$,
so $f+\ve \Omega_r\in \Theta_{2r}^2$.
\end{proof}

\begin{rem}
Let $f\in\rxh$. For $r\in\N$ let
$$\ve_r= \inf\left\{\ve\in\R\colon f+\ve \Omega_r\in \Theta_{2r}^2\right\}.$$
Then $(\ve_r)_r$ is a decreasing sequence, each $\ve_r$ can be computed with a semidefinite program,
and $f$ is trace-positive if and only if $\inf_r\ve_r\le 0$ by Theorem \ref{t:lass}.
\end{rem}

\subsection{Sum-of-squares perturbations of the tracial Motzkin polynomial}

We illustrate Theorem \ref{t:lass} on the polynomial
$m=x_2x_1^4x_2+x_2^2x_1^2x_2^2-3x_2x_1^2x_2+1\in\bih$ from Example \ref{ex:motzkin}.
In fact,  we show in Example \ref{ex:lass} below that for every $\ve>0$ there is $r\in\N$ such that
$m+\frac{\ve}{r!}x_1^{2r}$ is cyclically equivalent to a sum of hermitian squares.
This in particular improves the approximation of the commutative Motzkin polynomial 
by sums of squares given in \cite[Example 3.5]{LN}.
We start with a technical lemma.

\begin{lem}\label{l:calc}
Let $r\in 4\N+1$ and $\ve\ge \frac{r!}{(r-1)^{r-1}}$. Then the polynomial
$$p(t)=-\left(\frac{t-r}{r-1}\right)^{\frac{r-1}{2}}$$
satisfies
$$2p(t)\le t-3 \quad\text{and}\quad t p(t)^2\le 1+\frac{\ve}{r!}t^r$$
for all $t\ge0$.
\end{lem}

\begin{proof}
Observe that $p$ is concave on $\R$ (since $\frac{r-1}{2}$ is even), $p(1)=-1$ and $p'(1)=\frac12$.
Therefore $p(t)\le \frac12(t-3)$ for all $t\in\R$.

On $(0,\infty)$, the function $t\mapsto t(t-1)^{r-1}$ has precisely two local extrema,
a local maximum at $\frac{1}{r}$ and a local minimum at $1$. Therefore
$$t(t-1)^{r-1}-\frac{1}{r}\left(\frac{1}{r}-1\right)^{r-1}\le t^r$$
for all $t\ge0$.
Replacing $t$ with $\frac{t}{r}$ gives
$$
\left(\frac{t}{r}\right)\left(\left(\frac{t}{r}\right)-1\right)^{r-1}\le
\frac{1}{r}\left(\frac{1}{r}-1\right)^{r-1}+\left(\frac{t}{r}\right)^{r},
$$
and after multiplication by $\frac{r^r}{(r-1)^{r-1}}$ we obtain
$$
t\left(\frac{t-r}{r-1}\right)^{r-1}\le 1
+\frac{t^r}{(r-1)^{r-1}}.
$$
Therefore $t p(t)^2\le 1+\frac{\ve}{r!}t^r$ holds for all $t\ge0$.
\end{proof}

\begin{exa}\label{ex:lass}
Let $M=t_1^4t_2^2+t_1^2t_2^4-3t_1^2t_2^2+1\in\R[t_1,t_2]$ be the Motzkin polynomial.
If $r\in 4\N+1$ and $\ve\ge \frac{r!}{(r-1)^{r-1}}$,
then $M+\frac{\ve}{r!}t_1^{2r}$ is a sum of squares in $\R[t_1,t_2]$.
Note that $\frac{r!}{(r-1)^{r-1}}$ decays exponentially towards $0$ as $r\to\infty$.

Indeed, let $p$ be as in Lemma \ref{l:calc}.
The univariate polynomials $t_1^4-3t_1^2-2t_1^2p(t_1^2)$
and
$1+\frac{\ve}{r!}t_1^{2r}-t_1^2p(t_1^2)^2$ are nonnegative,
and therefore sums of squares in $\R[t_1]$.
Then
\begin{align*}
M+\frac{\ve}{r!}t_1^{2r}
&=t_1^2t_2^4+(t_1^4-3t_1^2)t_2^2+\left(1+\frac{\ve}{r!}t_1^{2r}\right) \\
&=\Big(t_1^4-3t_1^2-2t_1^2p(t_1^2)\Big)t_2^2+
t_1^2\Big(t_2^2+p(t_1^2)\Big)^2
+\left(1+\frac{\ve}{r!}t_1^{2r}-t_1^2p(t_1^2)^2\right)
\end{align*}
is a sum of (at most five) squares in $\R[t_1,t_2]$.

Let $m\in\bih$ be a cyclically sorted noncommutative lift of $M$.
As in the proof of Corollary \ref{c:cyc} we conclude that
$$m+\frac{\ve}{r!}x_1^{2r} \in \Theta_{2r}^2$$
for $r\in 4\N+1$ and $\ve\ge \frac{r!}{(r-1)^{r-1}}$.
\hfill $\square$
\end{exa}


\end{document}